\documentclass[final,leqno]{amsart}
\usepackage{amsmath,amssymb,amsfonts,latexsym,stmaryrd}
\usepackage{graphicx,float,epsfig} 
\usepackage{mathrsfs} 
\usepackage{float} 
\usepackage{color}
\usepackage{setspace} 

\usepackage{soul}
\usepackage[dvipsnames]{xcolor}

\DeclareGraphicsExtensions{ .eps,.ps} \usepackage{subfig}
\usepackage{multirow}
\usepackage{url}
\usepackage[linesnumbered,algoruled,boxed]{algorithm2e}
\usepackage[colorlinks=false, linkcolor=blue,  citecolor=OliveGreen,  pdfstartview={}{}]{hyperref} 
\usepackage[dvipsnames]{xcolor}

\usepackage{amsmath}

\def\O{\Omega}

\newtheorem{prob}{Problem}
\renewcommand\sp{\mathop{\mathrm{Sp}}\nolimits}
\newtheorem{remark}{Remark}[section]
\newtheorem{lemma}{Lemma}[section]
\newtheorem{corollary}{Corollary}[section]
\newtheorem{thm}{Theorem}[section]

\newtheorem{definition}{Definition}
\usepackage{booktabs}
\usepackage{float}

\newcommand\bu{\boldsymbol{u}}

\newcommand\bw{\boldsymbol{w}}

\newcommand\bT{\boldsymbol{T}}

\def\CT{{\mathcal T}}

\newcommand\bPi{\boldsymbol{\Pi}}

\renewcommand\O{\Omega}

\renewcommand\sp{\mathop{\mathrm{sp}}\nolimits}

\newcommand{\vertiii}[1]{{\left\vert\kern-0.25ex\left\vert\kern-0.25ex\left\vert #1 
    \right\vert\kern-0.25ex\right\vert\kern-0.25ex\right\vert}}

\begin{document}

\title[VEM for the elasticity eigenproblem with small edges]
{A virtual element method for the elasticity spectral  problem allowing small edges}

\author{Danilo Amigo}
\address{GIMNAP-Departamento de Matem\'atica, Universidad del B\'io - B\'io, Casilla 5-C, Concepci\'on, Chile.}
\email{danilo.amigo2101@alumnos.ubiobio.cl}
\author{Felipe Lepe}
\address{GIMNAP-Departamento de Matem\'atica, Universidad del B\'io - B\'io, Casilla 5-C, Concepci\'on, Chile.}
\email{flepe@ubiobio.cl}
\thanks{The first and second authors were partially supported by
DIUBB through project 2120173 GI/C Universidad del B\'io-B\'io and ANID-Chile through FONDECYT project 11200529 (Chile).}

\author{Gonzalo Rivera}
\address{Departamento de Ciencias Exactas,
Universidad de Los Lagos, Casilla 933, Osorno, Chile.}
\email{gonzalo.rivera@ulagos.cl}
\thanks{The third author was partially supported by
 Universidad de Los Lagos through regular project R02/21.}




\begin{abstract}
In this paper we analyze a virtual element method for the two dimensional elasticity spectral problem allowing small edges. Under this approach, and with the aid of the theory of compact operators, we prove convergence of the proposed VEM and error estimates, where the influence of the Lam\'e constants  is presented. We present a series of numerical tests to assess the performance of the method where we analyze the effects of the Poisson ratio on the computation of the order of convergence, together with the effects of the stabilization term on the arising of spurious eigenvalues.
\end{abstract}

\maketitle

\section{Introduction}
\label{sec:intro}
The virtual element method (VEM), introduced in \cite{MR2997471} as an alternative to solve partial differential equations, has proved through time several applications to approximate accurately the solutions of different problems. In \cite{ABM2022} we find recent advances in the applications of VEM, which have been possible thanks to several works developed in fluid problems \cite{beirao2019stokes}, elasticity problems \cite{MR3033033,MR3767447,MR3881593}, eigenvalue problems \cite{MR3867390, MR4284360, MR4229296, MR4050542, MR3340705}, among others. 

The VEM results to be attractive since its nature allows to discretize with different polygonal meshes, domains that can be difficult to mesh, for example, domains with cracks or nonconvex domains. Despite  the fact that some methods as the discontinuous Galerkin method (DG) allow to consider hanging nodes, those methods consider a triangle of the mesh, for instance, as a triangle but with an extra point that is a vertex of other triangle, whereas VEM considers this fact a vertex of a new polygon,  allowing a new treatment and discretization for the geometrical domain. Of course, VEM is simple to implement and reduces computational costs compared with some classic FEM, as for example, the discretization of fourth order elliptic problem. Although these interesting advantages, the research on VEM is in ongoing process, and more general methods involving virtual spaces have emerged.

One of the hypotheses that  \cite{MR2997471} show to perform the VEM analysis is that the polygons on the mesh must have sides (or faces) which are not allowed to be arbitrary small. This assumption has been relaxed in \cite{MR3714637, MR3815658} where, according to the theory developed in these references, it  is sufficient to require the star-shapedness of the polygonal elements of the mesh. This is clearly an important advantage for the VEM, bust there is a cost to pay, since for the best of the author's knowledge,  not any problem can be discretized with this new approach.

In first place, the VEM allowing small edges are constructed for subspaces of $H^1$ and for second order elliptic differential operators. The second is related  to the regularity of the functions, since according to \cite{MR3714637}, to use only star-shaped polygons, the regularity of the solution in order to obtain approximation properties must be such that $H^{1+s}$ with $s>1/2$. This is an essential restriction to use in a clean way the small edges approach. In this same line, the regularity will depend on the differential operator, the geometry of the domain, boundary conditions, etc.. Let us remark that  the VEM allowing small edges has been applied in some problems as \cite{ALR:22,MR4359996,MR4284360,MR4461634}, and the research is in progress.

In particular we are interested in the application of VEM allowing small edges on the linear elasticity equations. This research begun with the load problem analyzed in \cite{ALR:22}, where the two dimensional elasticity problem is analyzed in a convex domain with Lipschitz boundary.  The convexity of the domain is a key ingredient for the analysis, since the regularity of the solution lies precisely in the requirements of \cite{MR3714637}. Let us remark that if mixed boundary conditions are considered, the solution has less regularity due to the reentrant angles that may appear (see \cite{MR840970}) and the small edges framework still hold, but it is necessary to assume a further condition on the geometry, which is that the number of edges of the polygons must be bounded (see \cite{MR3714637,MR3815658}). Here the price to pay is more expensive and is reflected in the error estimate of the solution, which will depend strongly on a constant depending on the mesh size. A discussion on this subject can be found in  \cite{MR3714637}. 

This is a drawback that cannot be avoided and strongly deteriorates the elasticity eigenvalue problem, since it is not possible to ensure the convergence in norm of the respective solutions operators and hence, the spectral convergence. This is the reason why only Dirichlet boundary conditions (clamped conditions in particular) are considered to perform the analysis.

The paper is organized as follows: In section \ref{sec:model} we present the spectral problem of our interest and summarize some important properties related to the solution. The continuous solution operator is presented, the regularity of the eigenfunctions, and the corresponding spectral characterization. The core of the manuscript begins in section \ref{sec:vem}, where the virtual element method is presented. In this context, we introduce the necessary ingredients to perform the analysis for the small edges scheme. We present the discrete eigenvalue problem and with the aid of the results proved in \cite{ALR:22} together with the classic theory of \cite{MR1115235}  we prove convergence in norm for the operators, spectral convergence, and error estimates for eigenvalues and eigenfunctions. Finally, in section \ref{sec:numerics} we present a complete and rigorous computational analysis of the method. This section presents the computation of eigenvalues, analysis of spurious eigenvalues with respect to the stabilization terms and its influence, and computation order of convergence  for the eigenvalues.

\section{Model problem}
\label{sec:model}
Let $\O\subset\mathbb{R}^2$ be a open, bounded and convex domain with Lipschitz boundary $\partial\O$. The model problem is the following: Find $\kappa\in\mathbb{R}$ and the displacement $\textbf{w}$ such that 
\begin{equation}
\label{loadspect}
\left\{\begin{array}{cccc}
\textbf{div}(\boldsymbol{\sigma}(\textbf{w})) &=& -\varrho \kappa \textbf{w} \quad \text{in} \; \Omega, \\
\textbf{w} &=&  0 \quad \text{on} \;  \partial \Omega, 
\end{array}\right.
\end{equation}

A variational formulation for \eqref{loadspect} is the following.
\begin{prob}
\label{eigen_continuo}
Find $(\kappa, \textbf{w}) \in \mathbb{R} \times \mathbf{H}_{0}^{1}(\Omega)$ with  $\textbf{w} \neq 0$ such that  
\begin{equation*}
a(\textbf{w}, \textbf{v}) = \kappa b(\textbf{w}, \textbf{v}) \quad \forall \textbf{v} \in \mathbf{H}_{0}^{1}(\Omega),
\end{equation*}
\end{prob} 
\noindent where the symmetric and continuous bilinear forms $a(\cdot,\cdot)$ and $b(\cdot,\cdot)$ are defined by
\begin{equation*}\label{a:spec}
a :  \mathbf{H}_{0}^{1}(\Omega) \times \mathbf{H}_{0}^{1}(\Omega) \longrightarrow \mathbb{R}, \quad a(\textbf{u}, \textbf{v}) := \displaystyle{\int_{\Omega}} \boldsymbol{\sigma}(\textbf{u}) : \boldsymbol{\varepsilon}(\textbf{v}) \quad \forall \textbf{u}, \textbf{v} \in \mathbf{H}_{0}^{1}(\Omega),
\end{equation*}
and 
\begin{equation*}\label{b}
b :  \mathbf{H}_{0}^{1}(\Omega) \times \mathbf{H}_{0}^{1}(\Omega) \longrightarrow \mathbb{R}, \quad b(\textbf{u}, \textbf{v}) := \displaystyle{\int_{\Omega} \varrho\textbf{u}\cdot\textbf{v}}  \quad \forall \textbf{u}, \textbf{v} \in \mathbf{H}_{0}^{1}(\Omega).
\end{equation*}

From  Korn's inequality, the coercivity of $a(\cdot,\cdot)$ on $\mathbf{H}^1_0(\O)$ is direct. This allows us to introduce the solution operator $\textbf{\text{T}}$, defined by \begin{equation*}
\textbf{T}:  \mathbf{H}_{0}^{1}(\Omega)\longrightarrow\mathbf{H}_{0}^{1}(\Omega), \quad
\textbf{f}\longmapsto\textbf{Tf} = \widetilde{\textbf{w}},
\end{equation*} 
where  $\widetilde{\textbf{w}}\in\mathbf{H}_{0}^{1}(\Omega)$ is the solution of the following source problem
\begin{equation*}\label{fuente}
a(\widetilde{\textbf{w}}, \textbf{v}) = b(\textbf{f}, \textbf{v}), \qquad \forall \; \textbf{v} \in \mathbf{H}_{0}^{1}(\Omega), 
\end{equation*}
which is well posed due  Lax-Milgram's lemma, implying  that  $\textbf{T}$ is well defined and satisfies
\begin{equation*}
\|\textbf{Tf}\|_{1, \Omega} = \|\widetilde{\textbf{w}}\|_{1, \Omega} \lesssim\|\textbf{f}\|_{0, \Omega},
\end{equation*} 
where the hidden constant depends on $\O$. It is easy to check that $\textbf{T}$ is selfadjoint with respect to  $a(\cdot, \cdot)$.  Moreover, from the compact embedding of  $\mathbf{H}_{0}^{1}(\Omega)$ onto  $\textbf{L}^{2}(\Omega)$ we have that  $\textbf{T}$ is compact.
\begin{remark} Let  $(\textbf{w},\kappa) \in  \mathbf{H}^{1}_{0}(\Omega)\times \mathbb{R}$ be the  solution of Problem \ref{eigen_continuo}. Then, if  $\bu \in \mathbf{H}^{1}_{0}(\Omega)$ is such that  $\textbf{T}\textbf{u} = \textbf{w}$, then for each $\textbf{v} \in \mathbf{H}^{1}_{0}(\Omega)$ there holds
\begin{center}
$a(\textbf{w}, \textbf{v}) = b(\textbf{u}, \textbf{v}) = \dfrac{\kappa}{\kappa}b(\textbf{u}, \textbf{v}) = \dfrac{1}{k}a(\textbf{u}, \textbf{v}) = a(\eta \textbf{u}, \textbf{v})$, \quad $\eta := \dfrac{1}{\kappa}$,
\end{center}
 implying  $\textbf{Tu} = \eta \textbf{u}$. Hence,  $(\textbf{w},\kappa)\in\mathbf{H}^{1}_{0}(\Omega)\times \mathbb{R}$ solves  Problem \ref{eigen_continuo} if and only if  $(\textbf{w},\eta)$ is an eigenpair of  $\textbf{T}$. 
\end{remark}

Let us recall the following regularity result (see \cite{MR840970} for instance).
\begin{lemma}
\label{reg2} 
Let $\Omega\subset\mathbb{R}^2$ be an open, bounded, and convex domain. If  $(\textbf{w},\kappa)\in\mathbf{H}_0^1(\O)\times\mathbb{R}$ solves Problem  \ref{eigen_continuo}, then $\textbf{w} \in \mathbf{H}^{2}(\Omega)$ and the following estimate holds
\begin{center}
$\|\textbf{w}\|_{2, \Omega} \lesssim \|\textbf{w}\|_{0, \Omega}$,
\end{center}
where the hidden constant depends on the eigenvalue $\kappa$.
\end{lemma}
We end this section with the spectral characterization of  $\textbf{T}$
\begin{thm}  
The spectrum of $\bT$    satisfies $\sp(\bT)=\{0\}\cup\{\mu_k\}_{k\in\mathbb{N}}$, where $\{\mu_k\}_{k\in\mathbb{N}}$ is a sequence of positive eigenvalues such that $\mu_k\rightarrow 0$ as $k\rightarrow+\infty$.
\end{thm}

\section{The virtual element method}
\label{sec:vem}
In the present section we introduce the virtual element method that we consider to approximate the solution of Problem \ref{eigen_continuo}. To do this task, we will consider a more relaxed   conditions compared with those introduced in  \cite{MR2997471} for the classic VEM, where there is not possible to assume more general polygonal meshes allowing  arbitrary edges, more precisely, small edges. Hence, and inspired in  \cite{MR3714637}, if $\{\mathcal{T}_h\}_{h>0}$ represents a family of polygonal meshes to discretize $\O$,  $E\in\CT_h$ is an arbitrary element of the mesh, and  $h := \underset{E \in \mathcal{T}_{h}}{\max} \; h_{E}$ represents the mesh size,   we assume the following assumption on $\CT_h$:
\begin{itemize}
\item[\textbf{A1.}] There exists  $\gamma \in \mathbb{R}^{+}$ such that each polygon $E \in \{\mathcal{T}_{h}\}_{h>0}$ is star-shaped with respect to a ball  $B_{E}$ with center  $\textbf{x}_{E}$ and radius  $\rho_{E} \geq \gamma h_{E}$.
\end{itemize}
Let us write the bilinear form $a(\cdot, \cdot)$ and the functional $F(\cdot)$ as follows
\begin{center}
$a(\textbf{u}, \textbf{v}) = \displaystyle{\sum_{E \in \mathcal{T}_{h}} a^{E}(\textbf{u}, \textbf{v})}$ \quad where $a^{E}(\textbf{u}, \textbf{v}) := \displaystyle{\int_{E} \boldsymbol{\sigma}(\textbf{u}) : \boldsymbol{\varepsilon}(\textbf{v})}$ \quad $\forall \; \textbf{u}, \textbf{v} \in \mathbf{H}^{1}_{0}(\Omega)$,
\end{center}
\begin{center}
$b(\textbf{u}, \textbf{v}) = \displaystyle{\sum_{E \in \mathcal{T}_{h}} b^{E}(\textbf{u}, \textbf{v})}$, with $b^{E}(\textbf{u}, \textbf{v}) = \displaystyle{\int_{E} \varrho\textbf{u}\cdot\textbf{v}}$ \quad $\forall \textbf{u}, \textbf{v} \in \mathbf{H}_{0}^{1}(\Omega)$.
\end{center}

\subsection{Virtual spaces}
Now we introduce the virtual spaces of our interest. Following \cite{MR4567234} and \cite{MR3714637}, we introduce the following local spaces
\begin{center}
$\mathbb{B}_{\partial E} := \{\textbf{v}_{h} \in \boldsymbol{\mathcal{C}}^{0}(\partial E) \; : \; \textbf{v}_{h}\lvert_{e} \in \mathbb{P}_{k}(e) \; \forall e \subset \partial E\}$,
\end{center}
\begin{center}
$\boldsymbol{\mathcal{W}}_{h}^{E} := \{\textbf{v}_{h} \in \mathbf{H}^{1}(E) \; : \; \Delta \textbf{v}_{h} \in [\mathbb{P}_{k}(E)]^{2} \; \text{y} \; \textbf{v}_{h}\lvert_{\partial E} \; \in \mathbb{B}_{\partial E}\}$.
\end{center}
For each,  $E \in \{\mathcal{T}_{h}\}_{h>0}$, we introduce the projection $\boldsymbol{\Pi}_{k,E} : \boldsymbol{\mathcal{W}}_{h}^{E} \longrightarrow [\mathbb{P}_{k}(E)]^{2},$
defined for every  $\textbf{v}_{h} \in \boldsymbol{\mathcal{W}}_{h}^{E}$ as the solution of 
\begin{equation*}
\label{elastt}
\left\{\begin{array}{ccl}
\displaystyle{\int_{E} \boldsymbol{\varepsilon}(\boldsymbol{\Pi}_{k,E}\textbf{v}_{h}):\boldsymbol{\varepsilon}(\boldsymbol{p})} &=& \displaystyle{\int_{E} \boldsymbol{\varepsilon}(\textbf{v}_{h}) : \boldsymbol{\varepsilon}(\boldsymbol{p})}  \quad \forall \; \boldsymbol{p} \in [\mathbb{P}_{k}(E)]^{2}, \\
\displaystyle{\int_{E} \text{rot}(\boldsymbol{\Pi}_{k,E}\textbf{v}_{h})} &=& \displaystyle{\int_{E} \text{rot}(\textbf{v}_{h})}, \\
\displaystyle{\int_{\partial E} \boldsymbol{\Pi}_{k,E}\textbf{v}_{h}} &=& \displaystyle{\int_{\partial E} \textbf{v}_{h}}.
\end{array}\right.
\end{equation*}


We define the local virtual space by 
\begin{center}
$\boldsymbol{\mathcal{V}}_{h}^{E} := \left\lbrace \textbf{v}_{h} \in \boldsymbol{\mathcal{W}}_{h}^{E} \; : \; \displaystyle{\int_{E} \boldsymbol{p}\cdot (\textbf{v}_{h} - \boldsymbol{\Pi}_{k,E}\textbf{v}_{h}) = 0, \forall \; \boldsymbol{p} \in [\mathbb{P}_{k}(E)]^{2}/[\mathbb{P}_{k-2}(E)]^{2}} \right\rbrace$,
\end{center}
where  the space $[\mathbb{P}_{k}(E)]^{2}/[\mathbb{P}_{k-2}(E)]^{2}$ denotes the polynomials in $[\mathbb{P}_{k}(E)]^{2}$ in   which are orthogonal to  $[\mathbb{P}_{k-2}(E)]^{2}$ with respect to the $\textbf{L}^{2}(E)$ product.
We choose the same degrees of freedom as those  in \cite[Section 4.1]{MR2997471} for the local virtual space defined above.

Now we are in position to introduce the global virtual space which we define by 
\begin{equation*}
\boldsymbol{\mathcal{V}}_{h} := \{ \textbf{v}_{h} \in \mathbf{H}_{0}^{1}(\Omega) \; : \; \textbf{v}_{h}\lvert_{E} \; \in \boldsymbol{\mathcal{V}}_{h}^{E}\}.
\end{equation*}
Let us introduce the following stabilization term  $S^{E}(\cdot, \cdot)$ defined for $\mathbf{u}_{h},\mathbf{v}_{h}\in\boldsymbol{\mathcal{V}}_h$ by 
\begin{equation*}
S^{E}(\textbf{u}_{h}, \textbf{v}_{h}) := h_{E}\displaystyle{\int_{\partial E} \partial_{s}\textbf{u}_{h}\cdot\partial_{s}\textbf{v}_{h}},
\end{equation*}
which corresponds to a scaled inner product between $\partial_{s}\textbf{u}_{h}$ and $\partial_{s}\textbf{v}_{h}$ in $\textbf{L}^{2}(\partial E)$. Let us introduce the discrete bilinear form  $a_{h}(\cdot, \cdot):\boldsymbol{\mathcal{V}}_h\times \boldsymbol{\mathcal{V}}_h\rightarrow\mathbb{R}$ defined by 
\begin{equation*}
a_{h}(\textbf{u}_{h}, \textbf{v}_{h}) := \displaystyle{\sum_{E \in \mathcal{T}_{h}} \left[a^{E}(\boldsymbol{\Pi}_{k,E}\textbf{u}_{h}, \boldsymbol{\Pi}_{k,E}\textbf{v}_{h}) + S^{E}(\textbf{u}_{h} - \boldsymbol{\Pi}_{k,E}\textbf{u}_{h}, \textbf{v}_{h} - \boldsymbol{\Pi}_{k,E}\textbf{v}_{h})\right]}.
\end{equation*}

Now, the local discrete bilinear forms are the following 
\begin{equation*}
a_{h}^{E}(\textbf{u}_{h}, \textbf{v}_{h}) := a^{E}(\boldsymbol{\Pi}_{k,E}\textbf{u}_{h}, \boldsymbol{\Pi}_{k,E}\textbf{v}_{h}) + S^{E}(\textbf{u}_{h} - \boldsymbol{\Pi}_{k,E}\textbf{u}_{h}, \textbf{v}_{h} - \boldsymbol{\Pi}_{k,E}\textbf{v}_{h}) \quad \forall \textbf{u}_{h}, \textbf{v}_{h} \in \boldsymbol{\mathcal{V}}_{h}^{E}, 
\end{equation*}
and 
\begin{equation*}
b_{h}^{E}(\textbf{u}_{h}, \textbf{v}_{h}):= b^{E}(\boldsymbol{\Pi}_{k,E}^{0}\textbf{u}_{h}, \boldsymbol{\Pi}_{k,E}^{0}\textbf{v}_{h}) \quad \forall \textbf{u}_{h}, \textbf{v}_{h} \in \boldsymbol{\mathcal{V}}_{h}^{E},
\end{equation*}
Let us remark that  $b^{E}_{h}(\cdot, \cdot)$ is directly computable from the degrees of freedom.

Finally we introduce the global discrete bilinear forms as follows
\begin{equation*}
a_{h}(\textbf{u}_{h},\textbf{v}_{h}) := \displaystyle{\sum_{E \in \mathcal{T}_{h}} a_{h}^{E}(\textbf{u}_{h},\textbf{v}_{h})} \quad\text{and}\quad b_{h}(\textbf{u}_{h},\textbf{v}_{h}) := \displaystyle{\sum_{E \in \mathcal{T}_{h}} b_{h}^{E}(\textbf{u}_{h},\textbf{v}_{h})},
\end{equation*}
which allows us to define the VEM discretization of Problem \ref{eigen_continuo}.
\begin{prob}\label{discretelast}
Find $(\kappa_{h}, \textbf{w}_{h}) \in \mathbb{R} \times \boldsymbol{\mathcal{V}}_{h}$ with  $\textbf{w}_{h} \neq 0$ such that
\begin{equation*}
a_{h}(\textbf{w}_{h}, \textbf{v}_{h}) = \kappa_{h}b_{h}(\textbf{w}_{h}, \textbf{v}_{h}) \quad \forall \textbf{v}_{h} \in \boldsymbol{\mathcal{V}}_{h}.
\end{equation*}
\end{prob}

To show that $a_{h}(\cdot,\cdot)$ is coercive, we recall some results (see \cite{ALR:22} for details).

\begin{corollary}\label{cor}
Assume that \textbf{A1} holds. Then, the following estimate holds
\begin{equation*}
|\textbf{v}_{h}|_{1,E} \lesssim \max\{\lambda_{S}\mu_{S},1\}\left(h_{E}^{-1}\vertiii{\textbf{v}_{h}}_{k,E} + h_{E}^{1/2}\|\partial_{s}\textbf{v}_{h}\|_{0,\partial E}\right) \quad \forall \textbf{v}_{h} \in \boldsymbol{\mathcal{V}}_{h},
\end{equation*}
where the hidden constant depends on $\rho_{E}$ and $k$, and not on $h_{E}$.
\end{corollary}

\begin{lemma}\label{corx}
The following estimate holds
\begin{equation*}
\vertiii{\textbf{v}_{h}}_{k,E}^{2} \lesssim h_{E}\displaystyle{\sum_{e \in \mathcal{E}_{E}} \|\boldsymbol{\Pi}_{k-1,e}\textbf{v}_{h}\|_{0,e}^{2}} \quad \forall \textbf{v}_{h} \in \boldsymbol{\mathcal{V}}_{h} \; \text{such that} \; \boldsymbol{\Pi}_{k,E}\textbf{v}_{h} = \textbf{0},
\end{equation*}
where the hidden constant is independent on $h_{E}$.
\end{lemma}

\begin{lemma}\label{corxx}
The following estimate holds
\begin{equation*}
\|\textbf{v}_{h}\|_{0,\partial E} \lesssim h_{E}\|\partial_{s}\textbf{v}_{h}\|_{0,\partial E},
\end{equation*}
for all $\textbf{v}_{h} \in \mathbb{B}_{\partial E}$ that vanishes at some point of $\partial E$, and the hidden constant depends only on $k$.
\end{lemma}

\begin{remark}
Let $\textbf{v} \in \mathbf{H}^{1}(E)$ such that $\boldsymbol{\Pi}_{k,E}\textbf{v} = \textbf{0}$. Then, applying Corollary \ref{cor}, Lemmas \ref{corx} and \ref{corxx}, we derive
\begin{equation*}
|\textbf{v}|_{1,E} \lesssim \max\{\lambda_{S}\mu_{S}^{-1},1\}h_{E}^{1/2}\|\partial_{s}\textbf{v}\|_{0,\partial E},
\end{equation*}
where the hidden constant is independent on $h_{E}$. Combining this with the fact that $\boldsymbol{\Pi}_{k,E}(\textbf{v} - \boldsymbol{\Pi}_{k,E}\textbf{v}) = \textbf{0}$ and applying triangular inequality, we obtain for $\textbf{v}_{h} \in \boldsymbol{\mathcal{V}}_{h}$
\begin{equation*}
|\textbf{v}_{h}|_{1,E}^{2} \lesssim |\boldsymbol{\Pi}_{k,E}\textbf{v}_{h}|_{1,E}^{2} + |\textbf{v}_{h} - \boldsymbol{\Pi}_{k,E}\textbf{v}_{h}|_{1,E}^{2} 
\lesssim \max\{\lambda_{S}^{2}\mu_{S}^{2},\mu_{S}^{-1},1\}a_{h}^{E}(\textbf{v}_{h},\textbf{v}_{h}).
\end{equation*}
Finally, taking summation over $E \in \mathcal{T}_{h}$, we obtain
\begin{equation*}
|\textbf{v}_{h}|_{1,\O}^{2} \lesssim \max\{\lambda_{S}^{2}\mu_{S}^{2},\mu_{S}^{-1},1\}a_{h}(\textbf{v}_{h},\textbf{v}_{h}).
\end{equation*}
This show that $a_{h}(\cdot,\cdot)$ is coercive in $\boldsymbol{\mathcal{V}}_{h}$.
\end{remark}

Now, thanks to the coercivity of $a_{h}(\cdot,\cdot)$ in  $\boldsymbol{\mathcal{V}}_{h}$, Problem \ref{discretelast} is well posed and hence, 
we are allowed to  introduce the discrete solution operator $\textbf{T}_h$, defined by 
\begin{equation*}
\textbf{T}_{h}: \mathbf{H}_{0}^{1}(\Omega)\longrightarrow\boldsymbol{\mathcal{V}}_{h}  \quad
\textbf{f}\longmapsto\textbf{T}_{h}\textbf{f} = \widetilde{\textbf{w}}_{h},
\end{equation*} 
such that  $\widetilde{\textbf{w}}_{h}$ is the unique solution of the following discrete source problem: Given $\textbf{f}\in\mathbf{L}^2(\O)$, find $\widetilde{\textbf{w}}_{h}\in\boldsymbol{\mathcal{V}}_{h}$ such that 
\begin{equation*}\label{fuentedisc}
a_{h}(\widetilde{\textbf{w}}_{h}, \textbf{v}_{h}) = b_{h}(\textbf{f}, \textbf{v}_{h})\qquad\forall\textbf{v}_h\in\boldsymbol{\mathcal{V}}_{h}.
\end{equation*}

Observe that  $\textbf{T}_{h}$    is selfadjoint with respect to  $a_{h}(\cdot, \cdot)$ and that is well defined by Lax-Milgram's lemma. Also,   we observe that  $(\textbf{w}_{h},\kappa_h)\in\boldsymbol{\mathcal{V}}_{h}\times\mathbb{R}$ 
solves Problem \ref{discretelast} if and only if $( \textbf{w}_{h},\eta_h)$ is an eigenpair of  $\textbf{T}_{h}$, i.e., 
\begin{equation*}
\textbf{T}_{h}\textbf{w}_{h} = \eta_{h}\textbf{w}_{h}, \quad \text{with} \quad \eta_{h} = \dfrac{1}{\kappa_{h}}.
\end{equation*}

Finally we present the spectral characterization of $\textbf{T}_{h}$.
\begin{thm}\label{spectchar}
The spectrum of  $\textbf{T}_{h}$ consists in  $M_{h} := \dim(\boldsymbol{\mathcal{V}}_{h})$ eigenvalues  with a certain multiplicity. Moreover, all these eigenvalues are real positive numbers. 
\end{thm}

\subsection{Technical results}
Now we will summarize some technical results that allows us to perform the analysis. All these results are available in \cite{ALR:22}  for the source problem, but are 
also valid for the spectral problem. The relevance of the forthcoming results yields in the fact that all the estimates show a clear dependence on the Lam\'e coefficient $\lambda_S$.
\begin{lemma}\label{lema43}
Assume that $\textbf{u} \in \mathbf{H}^{\ell + 1}(\Omega)$, $1 \leq \ell \leq k$. Then, there holds
\begin{equation*}
\displaystyle{\sum_{E \in \mathcal{T}_{h}} S^{E}(\textbf{u}_{h} - \boldsymbol{\Pi}_{k,E}\textbf{u}_{h},\textbf{u}_{h} - \boldsymbol{\Pi}_{k,E}\textbf{u}_{h})} \lesssim \mathcal{C}(\lambda_{S},\mu_{S})h^{2\ell}|\textbf{u}|_{\ell + 1,\Omega}^{2},
\end{equation*}
where $\mathcal{C}(\lambda_{S},\mu_{S})$ is a positive constant depending on the Lam\'e coefficients, and is as in \cite[Lemma 3.16]{ALR:22}.
\end{lemma}

\begin{thm}\label{teorema42}
Assume that  $\textbf{u}\in \mathbf{H}^{\ell + 1}(\Omega)$ for  $1 \leq \ell \leq k$. Then, there holds 
\begin{equation*}
|\textbf{u} - \textbf{u}_{h}|_{1,\Omega} + |\textbf{u} - \boldsymbol{\Pi}_{k,h}\textbf{u}_{h}|_{1,h} + |\textbf{u} - \boldsymbol{\Pi}_{k,h}^{0}\textbf{u}|_{1,h} \lesssim K(\lambda_{S},\mu_{S})h^{\ell}|\textbf{u}|_{\ell+1, \Omega},
\end{equation*}
where $K(\lambda_{S},\mu_{S})$ is a positive constant depending on the Lam\'e coefficients, and is as in \cite[Theorem 3.2]{ALR:22}.
\end{thm}


\begin{thm}\label{teorema43}
Assume that  $\textbf{u} \in \mathbf{H}^{\ell + 1}(\Omega)$, $1 \leq \ell \leq k$. Then
\begin{equation*}
\|\textbf{u} - \textbf{u}_{h}\|_{0,\Omega} \lesssim \mathfrak{R}(\lambda_{S},\mu_{S})h^{\ell + 1}|\textbf{u}|_{\ell + 1,\Omega},
\end{equation*}
where $\mathfrak{R}(\lambda_{S},\mu_{S})$ is a positive constant depending on the Lam\'e coefficients, which is defined  in \cite[Theorem 3.3]{ALR:22}.
\end{thm}

\begin{thm}\label{teorema44}
Assume that  $\textbf{u} \in \mathbf{H}^{\ell + 1}(\Omega)$, for $1 \leq \ell \leq k$.  Then there holds
\begin{equation*}
\|\textbf{u} - \boldsymbol{\Pi}_{k,h}^{0}\textbf{u}_{h}\|_{0,\Omega} + \|\textbf{u} - \boldsymbol{\Pi}_{k,h}\textbf{u}_{h}\|_{0,\Omega} \lesssim \mathfrak{C}(\lambda_{S},\mu_{S})h^{\ell + 1}|\textbf{u}|_{\ell + 1,\Omega},
\end{equation*}
where $\mathfrak{C}(\lambda_{S},\mu_{S})$ is a positive constant depending on the Lam\'e coefficients, which is defined  in \cite[Theorem 3.4]{ALR:22}.
\end{thm}

%
%

We begin with the following error estimate, which gives us an error estimate for eigenfunctions in $\textbf{L}^{2}$-norm. The proof of this result is based in a duality argument, which for our case, we adapt from  \cite[Theorem 3.3]{MR3197278}
\begin{thm}
For all $\textbf{f} \in \mathcal{E}$, if $\textbf{T}\textbf{f} = \textbf{u}$ and $\textbf{T}_{h}\textbf{f} = \textbf{u}_{h}$, we have
\begin{equation*}
\|\textbf{u} - \textbf{u}_{h}\|_{0,\O} \lesssim \mathfrak{D}(\lambda_{S},\mu_{S})h^{2}\|\textbf{u}\|_{2,\O},
\end{equation*}
where the hidden constant is independent of $h$ and $\mathfrak{D}(\lambda_{S},\mu_{S})$ is a positive constant depending on the Lam\'e coefficients.
\end{thm}

\begin{proof}
Let $\boldsymbol{\Phi} \in \mathbf{H}_{0}^{1}(\O)$ the unique solution of the problem
\begin{equation*}
a(\boldsymbol{\Phi},\textbf{v}) = b(\textbf{u} - \textbf{u}_{h}, \textbf{v}) \quad \forall \textbf{v} \in \mathbf{H}_{0}^{1}(\O).
\end{equation*}  
Then, if  we set  $\textbf{v} = \textbf{u} - \textbf{u}_{h}$ on the above problem, we have 
\begin{equation*}
\varrho\|\textbf{u} - \textbf{u}_{h}\|_{0,\O}^{2} = a(\textbf{u} - \textbf{u}_{h},\boldsymbol{\Phi}) 
= a(\textbf{u} - \textbf{u}_{h}, \boldsymbol{\Phi} - \textbf{I}_{k,h}\boldsymbol{\Phi}) + a(\textbf{u} - \textbf{u}_{h},\textbf{I}_{k,h}\boldsymbol{\Phi}).
\end{equation*}
Our task is to estimate the two terms in the right-hand side. For the first term, we have 
\begin{equation}\label{cota11}
\begin{split}
a(\textbf{u} - \textbf{u}_{h}, \boldsymbol{\Phi} - \textbf{I}_{k,h}\boldsymbol{\Phi}) &\lesssim \max\{\lambda_{S},\mu_{S}\}|\textbf{u} - \textbf{u}_{h}|_{1,\O}|\boldsymbol{\Phi} - \textbf{I}_{k,h}\boldsymbol{\Phi}|_{1,\O} \\
&\lesssim \max\{\lambda_{S}^{2}\mu_{S}^{-1},\lambda_{S},\mu_{S}\}K(\lambda_{S},\mu_{S})h^{2}\|\textbf{u}\|_{2,\O}\|\boldsymbol{\Phi}\|_{2,\O},
\end{split}
\end{equation} 
where, in the first inequality, we use the continuity of $a(\cdot,\cdot)$, Theorem \ref{teorema42} and \cite[Lemma 3.10]{ALR:22}. 

On the other hand, we have the following error equation

\begin{equation*}
\begin{split}
a(\textbf{u} - \textbf{u}_{h},\textbf{I}_{k,h}\boldsymbol{\Phi}) &= a(\textbf{u},\textbf{I}_{k,h}\boldsymbol{\Phi}) - a(\textbf{u}_{h},\textbf{I}_{k,h}\boldsymbol{\Phi}) \\
&= \underbrace{a_{h}(\textbf{u}_{h},\textbf{I}_{k,h}\boldsymbol{\Phi}) - a(\textbf{u}_{h}, \textbf{I}_{k,h}\boldsymbol{\Phi})}_{B_{1}} + \underbrace{b(\textbf{f}, \textbf{I}_{k,h}\boldsymbol{\Phi}) - b_{h}(\textbf{f},\textbf{I}_{k,h}\boldsymbol{\Phi})}_{B_{2}}.
\end{split}
\end{equation*}
To estimate  $B_{1}$, from the definition of $a_{h}(\cdot,\cdot)$, the continuity of $a(\cdot,\cdot)$, \cite[Lemmas 3.10, 3.15 and 3.16]{ALR:22} and Theorem \ref{teorema42}, we have
\begin{multline*}
B_{1} = \displaystyle{\sum_{E \in \mathcal{T}_{h}} a^{E}(\boldsymbol{\Pi}_{k,E}\textbf{u}_{h} - \textbf{u}_{h}, \boldsymbol{\Pi}_{k,E}\textbf{I}_{k,E}\boldsymbol{\Phi} - \textbf{I}_{k,E}\boldsymbol{\Phi})} + \displaystyle{\sum_{E \in \mathcal{T}_{h}} S^{E}((I - \boldsymbol{\Pi}_{k,E})\textbf{u}_{h}, (I - \boldsymbol{\Pi}_{k,E})\textbf{I}_{k,E}\boldsymbol{\Phi})} \\ 
\lesssim \max\{\lambda_{S},\mu_{S}\}\displaystyle{\sum_{E \in \mathcal{T}_{h}} |\textbf{u}_{h} - \boldsymbol{\Pi}_{k,E}\textbf{u}_{h}|_{1,E}|\textbf{I}_{k,E}\boldsymbol{\Phi} - \boldsymbol{\Pi}_{k,E}\textbf{I}_{k,E}\boldsymbol{\Phi}|_{1,E}} \\ 
+ \max\{\lambda_{S}\mu_{S}^{-1},1\}\mathfrak{C}(\lambda_{S},\mu_{S})^{1/2}h^{2}\|\textbf{u}\|_{2,\O}\|\boldsymbol{\Phi}\|_{2,\O} \\
\lesssim \max\{\lambda_{S}^{2}\mu_{S}^{-1},\lambda_{S},\mu_{S}\}K(\lambda_{S},\mu_{S})h^{2}\|\textbf{u}\|_{2,\O}\|\boldsymbol{\Phi}\|_{2,\O} + \max\{\lambda_{S}\mu_{S}^{-1},1\}\mathfrak{C}(\lambda_{S},\mu_{S})^{1/2}h^{2}\|\textbf{u}\|_{2,\O}\|\boldsymbol{\Phi}\|_{2,\O} \\
\lesssim Z(\lambda_{S},\mu_{S})h^{2}\|\textbf{u}\|_{2,\O}\|\boldsymbol{\Phi}\|_{2,\O},
\end{multline*}
where $Z(\lambda_{S},\mu_{S}) := \max\{\max\{\lambda_{S}^{2}\mu_{S}^{-1},\lambda_{S},\mu_{S}\}K(\lambda_{S},\mu_{S}), \max\{\lambda_{S}\mu_{S}^{-1},1\}\mathfrak{C}(\lambda_{S},\mu_{S})^{1/2}\}$.

Now, to bound $B_{2}$, thanks to \cite[Lemma 3.10]{ALR:22} and the stability of $\boldsymbol{\Pi}_{k,E}^{0}$ in $\textbf{L}^{2}(\O)$ norm, we have 
\begin{equation*}
\begin{split}
B_{2} &= \displaystyle{\sum_{E \in \mathcal{T}_{h}} b^{E}(\textbf{f},\textbf{I}_{k,h}\boldsymbol{\Phi}) - b^{E}(\boldsymbol{\Pi}_{k,E}^{0}\textbf{f}, \boldsymbol{\Pi}_{k,E}^{0}\textbf{I}_{k,h}\boldsymbol{\Phi})} \\
&= \displaystyle{\sum_{E \in \mathcal{T}_{h}} b^{E}(\textbf{f} - \boldsymbol{\Pi}_{k,E}^{0}\textbf{f},\textbf{I}_{k,h}\boldsymbol{\Phi} - \boldsymbol{\Pi}_{k,E}^{0}\textbf{I}_{k,h}\boldsymbol{\Phi})} \\
&\lesssim  \displaystyle{\sum_{E \in \mathcal{T}_{h}} \|\textbf{f} - \boldsymbol{\Pi}_{k,E}^{0}\textbf{f}\|_{0,\O}\|\textbf{I}_{k,h}\boldsymbol{\Phi} - \boldsymbol{\Pi}_{k,E}^{0}\textbf{I}_{k,h}\boldsymbol{\Phi}\|_{0,\O}} \lesssim \max\{\lambda_{S}\mu_{S}^{-1},1\}^{2}h^{2}\|\textbf{u}\|_{2,\O}\|\boldsymbol{\Phi}\|_{2,\O}.
\end{split}
\end{equation*}
Finally, using the additional regularity for $\boldsymbol{\Phi}$ and the estimate 
\begin{equation*}
\|\boldsymbol{\Phi}\|_{2,\O} \lesssim \|\textbf{u} - \textbf{u}_{h}\|_{0,\O},
\end{equation*}
we conclude the result, where the constant $\mathfrak{D}(\lambda_{S},\mu_{S})$ is defined by
\begin{equation*}
\mathfrak{D}(\lambda_{S},\mu_{S}) := \max\{\max\{\lambda_{S}^{2}\mu_{S}^{-1},\lambda_{S},\mu_{S}\}K(\lambda_{S},\mu_{S}), Z(\lambda_{S},\mu_{S}), \max\{\lambda_{S}\mu_{S}^{-1},1\}^{2}\}.
\end{equation*}
\end{proof}

\subsection{Spectral approximation and error estimates}

In this section, our task will be to show that the discrete operator $\textbf{T}_{h}$ converges to $\textbf{T}$. With this aim, and taking advantage of the compactness of $\textbf{T}$, we will prove that this convergence is precisely obtained in the norm $\|\cdot\|_{1,\O}$ in order to apply the theory of \cite{MR1115235}.  We remark that  the compact operator theory gives immediately the convergence of eigenfunctions and eigenvalues.

Let us begin  with the following result.
\begin{lemma}\label{convergencia1}
The following estimate holds
\begin{equation*}
\|(\textbf{T} - \textbf{T}_{h})\textbf{f}\|_{1,\Omega} \lesssim K(\lambda_{S},\mu_{S})h\|\textbf{f}\|_{1,\Omega} \quad \forall \textbf{f} \in \mathbf{H}_{0}^{1}(\Omega),
\end{equation*}
\end{lemma}
where the hidden constant is independent  $h$.
\begin{proof}
Note that, from Theorem \ref{teorema42} with $\ell = 1$ and applying Poincar\'e inequality, we deduce the following error estimate in $\|\cdot\|_{1,\Omega}$ norm 
\begin{equation*}\label{stimah1}
\|(\textbf{T} - \textbf{T}_{h})\textbf{f}\|_{1,\Omega} \lesssim K(\lambda_{S},\mu_{S})h|\textbf{T}\textbf{f}|_{2,\Omega}.
\end{equation*}
Then, applying Lemma \ref{reg2} and the continuity of $\textbf{T}$, we derive
\begin{equation*}
|\textbf{T}\textbf{f}|_{2,\O} \lesssim \|\textbf{f}\|_{0,\O} \leq \|\textbf{f}\|_{1,\O},
\end{equation*}
concluding  the proof.
\end{proof}

\begin{remark}
As a consequence of the previous corollary we have that isolated parts of $\sp(\bT)$ are precisely approximate by isolated parts of $\sp(\bT_h)$. This fact means that if $\kappa\neq 0$ is an isolated eigenvalue of $\bT_h$ with multiplicity $m$ and $\mathcal{E}$ denotes the associated invariant space for the corresponding eigenfunctions, then there exists $m$ eigenvalues of $\bT_h$ which we denote by $\kappa_h^{(1)},\ldots, \kappa_h^{(m)}$, all of them with their corresponding multiplicity and invariant space $\mathcal{E}_h$ associated to the corresponding discrete eigenfunctions, such that converge to $\kappa$.
\end{remark}

Now our aims is to obtain error estimates for the approximation of the eigenvalues and eigenfunctions. With this goal in mind, we recall the following definitions.
\begin{definition}We define the  gap $\widehat{\delta}$ between two closed subspaces $\boldsymbol{\mathcal{X}}$ e $\boldsymbol{\mathcal{Y}}$ of  $\mathbf{H}_{0}^{1}(\Omega)$ by 
\begin{equation*}
\widehat{\delta}(\boldsymbol{\mathcal{X}}, \boldsymbol{\mathcal{Y}}) := \max\{\delta (\boldsymbol{\mathcal{X}}, \boldsymbol{\mathcal{Y}}), \delta (\boldsymbol{\mathcal{Y}}, \boldsymbol{\mathcal{X}})\},
\end{equation*}
where
\begin{equation*}
\delta (\boldsymbol{\mathcal{X}}, \boldsymbol{\mathcal{Y}}) := \underset{\textbf{x} \in \boldsymbol{\mathcal{X}} : \|\textbf{x}\|_{1, \Omega} = 1}{\sup} \left\lbrace \underset{\textbf{y} \in \boldsymbol{\mathcal{Y}}}{\inf} \quad \|\textbf{x} - \textbf{y}\|_{1, \Omega} \right\rbrace.
\end{equation*}
\end{definition}

The following result provides error estimates for the eigenfunctions and eigenvalues of the elasticity spectral problem.
\begin{thm}
\label{convergencia3} 
The following estimates hold
\begin{itemize}
\item[i)] $\widehat{\delta}(\mathcal{E}, \mathcal{E}_{h}) \lesssim K(\lambda_{S},\mu_{S})\gamma_{h}$
\item[ii)] $\left\lvert \kappa - \kappa_{h}^{(i)} \right\lvert \lesssim K(\lambda_{S},\mu_{S})\gamma_{h}$, $i = 1, \ldots, m$
\end{itemize}
with 
\begin{equation*}
\gamma_{h} := \underset{\textbf{f} \in \mathcal{E} : \|\textbf{f}\|_{1, \Omega} = 1}{\sup} \; \|(\textbf{T} - \textbf{T}_{h})\textbf{f}\|_{1, \Omega},
\end{equation*}
and the hidden constants are independent of $h$.
\end{thm}

\begin{proof}
Thanks to Lemma \ref{convergencia1}, since  $\textbf{T}_{h}$ converges to  $\textbf{T}$ in norm,  the proof is a direct consequence of the compact operators theory  of  Babu\v{s}ka-Osborn (see \cite[Theorems 7.1 and 7.3]{MR1115235}).
\end{proof}

Theorem \ref{convergencia3} is a result with a preliminary error estimate for the eigenvalues. Nevertheless, we are able to improve the linear order of convergence of this result, proving a quadratic order of convergence for the eigenvalues. This is stated in the following result. 


\begin{thm}
The following estimate holds
\begin{equation*}
|\kappa - \kappa_{h}^{(k)}| \lesssim  \mathfrak{F}(\lambda_{S},\mu_{S})h^{2},
\end{equation*}
where the hidden constant is independent of $h$.
\end{thm}

\begin{proof}
Let  $(\kappa_{h}^{(i)},\textbf{w}_{h})\in\mathbb{R}\times\boldsymbol{\mathcal{V}}_h$  be the solution of  Problem \ref{discretelast} with  $\|\textbf{w}_{h}\|_{1,\Omega} = 1$. Thanks to the previous results, there exists $( \textbf{w},\kappa)\in\mathbf{H}_0^1(\O)\times\mathbb{R}$ solution of Problem \ref{eigen_continuo} such that 
\begin{equation*}
\|\textbf{w} - \textbf{w}_{h}\|_{1,\Omega} \lesssim K(\lambda_{S},\mu_{S})h,
\end{equation*}
where the hidden constant is independent of $h$. 

On the other hand, the following algebraic identity is straightforward
\begin{multline}
\label{eq:padra}
(\kappa_{h}^{(i)} - \kappa)b(\textbf{w}_{h}, \textbf{w}_{h}) = \underbrace{a(\textbf{w} - \textbf{w}_{h},\textbf{w} - \textbf{w}_{h}) - \kappa b(\textbf{w} - \textbf{w}_{h}, \textbf{w} - \textbf{w}_{h})}_{T_{1}}\\
 + \underbrace{\left[a_{h}(\textbf{w}_{h}, \textbf{w}_{h}) - a(\textbf{w}, \textbf{w})\right]}_{T_{2}} 
+ \kappa_{h}^{(i)}\underbrace{\left[b(\textbf{w}_{h}, \textbf{w}_{h}) - b_{h}(\textbf{w}_{h}, \textbf{w}_{h})\right]}_{T_{3}}.
\end{multline}

Now our aim is to estimate each of the contributions on the right hand side of \eqref{eq:padra}. From triangle inequality and the continuity of $a(\cdot, \cdot)$ and  $b(\cdot, \cdot)$ 
we have for the term $T_1$
\begin{multline}
\label{eq:TermT1}
|T_{1}| \leq |a(\textbf{w} - \textbf{w}_{h}, \textbf{w} - \textbf{w}_{h})| + \kappa|b(\textbf{w} - \textbf{w}_{h}, \textbf{w} - \textbf{w}_{h})| \\
\lesssim \max\{\lambda_{S},\mu_{S}\}|\textbf{w} - \textbf{w}_{h}|_{1,\Omega}^{2} + \kappa\varrho\|\textbf{w} - \textbf{w}_{h}\|_{1,\Omega}^{2} \\
\lesssim \max\{\lambda_{S},\mu_{S}\}\|\textbf{w} - \textbf{w}_{h}\|_{1,\Omega}^{2} + \kappa\varrho\|\textbf{w} - \textbf{w}_{h}\|_{1,\Omega}^{2} \\
\lesssim \max\{\lambda_{S},\mu_{S}\}\|\textbf{w} - \textbf{w}_{h}\|_{1,\Omega}^{2} 
\lesssim \mathcal{F}_{1}(\lambda_{S},\mu_{S})h^{2},
\end{multline}
with $\mathcal{F}_{1}(\lambda_{S},\mu_{S}) := \max\{\lambda_{S},\mu_{S}\}K(\lambda_{S},\mu_{S})^{2}$. 

On the other hand, invoking the definition of  $a_{h}(\cdot,\cdot)$, triangle inequality, Lemma \ref{lema43} and Theorem \ref{teorema42}, there holds for the term $T_2$
\begin{multline}
\label{eq:TermT2}
|T_{2}| = \left\lvert \displaystyle{\sum_{E \in \mathcal{T}_{h}} \left[ a_{h}^{E}(\textbf{w}_{h}, \textbf{w}_{h}) - a^{E}(\textbf{w}_{h}, \textbf{w}_{h})\right]} \right\lvert \\
= \left\lvert \displaystyle{\sum_{E \in \mathcal{T}_{h}} \left[ a_{h}^{E}(\textbf{w}_{h} - \boldsymbol{\Pi}_{k,E}\textbf{w}_{h}, \textbf{w}_{h} - \boldsymbol{\Pi}_{k,E}\textbf{w}_{h}) - a^{E}(\textbf{w}_{h} - \boldsymbol{\Pi}_{k,E}\textbf{w}_{h}, \textbf{w}_{h} - \boldsymbol{\Pi}_{k,E}\textbf{w}_{h})\right]} \right\lvert \\
\leq \left\lvert \displaystyle{\sum_{E \in \mathcal{T}_{h}} S^{E}(\textbf{w}_{h} - \boldsymbol{\Pi}_{k,E}\textbf{w}_{h}, \textbf{w}_{h} - \boldsymbol{\Pi}_{k,E}\textbf{w}_{h})}\right\lvert + \left\lvert \displaystyle{\sum_{E \in \mathcal{T}_{h}} a^{E}(\textbf{w}_{h} - \boldsymbol{\Pi}_{k,E}\textbf{w}_{h}, \textbf{w}_{h} - \boldsymbol{\Pi}_{k,E}\textbf{w}_{h})}\right\lvert \\
\lesssim \mathcal{C}(\lambda_{S},\mu_{S})h^{2} + \max\{\lambda_{S},\mu_{S}\}|\textbf{w}_{h} - \boldsymbol{\Pi}_{k,h}\textbf{w}_{h}|_{1,h}^{2} \\
\lesssim \mathcal{C}(\lambda_{S},\mu_{S})h^{2} + \max\{\lambda_{S},\mu_{S}\}\left(|\textbf{w}_{h} - \textbf{w}|_{1,\Omega} + |\textbf{w} - \boldsymbol{\Pi}_{k,h}\textbf{w}_{h}|_{1,h}\right)^{2} \\
\lesssim \mathcal{C}(\lambda_{S},\mu_{S})h^{2} + \max\{\lambda_{S},\mu_{S}\}K(\lambda_{S},\mu_{S})^{2}h^{2} \lesssim \mathcal{F}_{2}(\lambda_{S},\mu_{S})h^{2},
\end{multline}
where $\mathcal{F}_{2}(\lambda_{S},\mu_{S}) := \max\{\mathcal{C}(\lambda_{S},\mu_{S}), \max\{\lambda_{S},\mu_{S}\}(K(\lambda_{S},\mu_{S}))^{2}\}$.

Now, from the definition of  $\boldsymbol{\Pi}_{k,h}^{0}$, Theorem  \ref{teorema43}  and Theorem  \ref{teorema44}, the term $T_3$ is estimated as follows
\begin{multline}
\label{eq:TermT3}
|T_{3}| = \left\lvert \displaystyle{\sum_{E \in \mathcal{T}_{h}} \left[b_{h}^{E}(\textbf{w}_{h}, \textbf{w}_{h}) - b^{E}(\textbf{w}_{h}, \textbf{w}_{h})\right]} \right\lvert \\
=  \left\lvert \displaystyle{\sum_{E \in \mathcal{T}_{h}} b^{E}(\textbf{w}_{h} - \boldsymbol{\Pi}_{k,E}^{0}\textbf{w}_{h}, \textbf{w}_{h} - \boldsymbol{\Pi}_{k,E}^{0}\textbf{w}_{h})} \right\lvert 
\lesssim \displaystyle{\sum_{E \in \mathcal{T}_{h}} \|\textbf{w}_{h} - \boldsymbol{\Pi}_{k,E}^{0}\textbf{w}_{h}\|_{0,E}^{2}} \\
\lesssim \displaystyle{\sum_{E \in \mathcal{T}_{h}} \left(\|\textbf{w} - \textbf{w}_{h}\|_{0,E}^{2} + \|\textbf{w} - \bPi_{k,E}^{0}\textbf{w}_{h}\|_{0,E}^{2}\right)} \lesssim \mathcal{F}_{3}(\lambda_{S},\mu_{S})h^{2},
\end{multline}
where  $\mathcal{F}_{3}(\lambda_{S},\mu_{S}) := \max\{\mathfrak{R}(\lambda_{S},\mu_{S})^{2}, \mathfrak{C}(\lambda_{S},\mu_{S})^{2}\}$.

Finally, since  $\kappa_{h}^{(i)} \longrightarrow \kappa$ as $h\rightarrow 0$, then $\{\kappa_{h}^{(i)}\}_{h > 0}$ is a bounded sequence and hence, together with the  coercivity of  $a_{h}(\cdot, \cdot)$ on $\boldsymbol{\mathcal{V}}_{h}$ we obtain
\begin{equation}
\label{eq:Termbh}
b_{h}(\textbf{w}_{h}, \textbf{w}_{h})  = \dfrac{a_{h}(\textbf{w}_{h}, \textbf{w}_{h})}{\kappa_{h}^{(i)}} \geq \dfrac{C\|\textbf{w}_{h}\|_{1,\Omega}}{\kappa_{h}^{(i)}} \geq \widetilde{C} > 0.
\end{equation}
Hence, gathering \eqref{eq:TermT1}, \eqref{eq:TermT2}, \eqref{eq:TermT3},  \eqref{eq:Termbh}, and replacing these estimates in \eqref{eq:padra}  we  conclude the proof.
\end{proof}

\section{Numerical experiments}
\label{sec:numerics}
In the following section we present a number of numerical tests in order to assess the performance of the proposed method. The main goal is to observe the accuracy of the small edges approach for the elasticity spectral problem in different computational domains and boundary conditions. The results that we report have been obtained with a MATLAB code. 
Through this section we will consider different polygonal meshes allowing small edges (i.e., satisfying only Assumption \textbf{A1}) and different values of the Poison ratio $\nu$. This last parameter is important since the Lam\'e coefficients are computed with that aid of this parameter  according to the following definitions
\begin{equation*}
\displaystyle\mu_S=\frac{\Lambda}{2(1+\nu)}\quad\text{and}\quad\lambda_S=\frac{\Lambda\nu}{(1+\nu)(1-2\nu)},
\end{equation*}
where clearly $\lambda_S$ blows up when $\nu\rightarrow 1/2$. This will lead to a loss of order of convergence, as we expect.

We begin our tests considering a convex domain.
\subsection{Unit square} 
In this test the computational domain is $\O=(0,1)^2$ with null boundary conditions on $\partial\Omega$., i.e, $\boldsymbol{w}=\boldsymbol{0}$. To discretize this domain we consider polygonal meshes as the ones presented in Figure \ref{fig:meshesx}.
\begin{figure}[H]
	\begin{center}
			\centering\includegraphics[height=5cm, width=6cm]{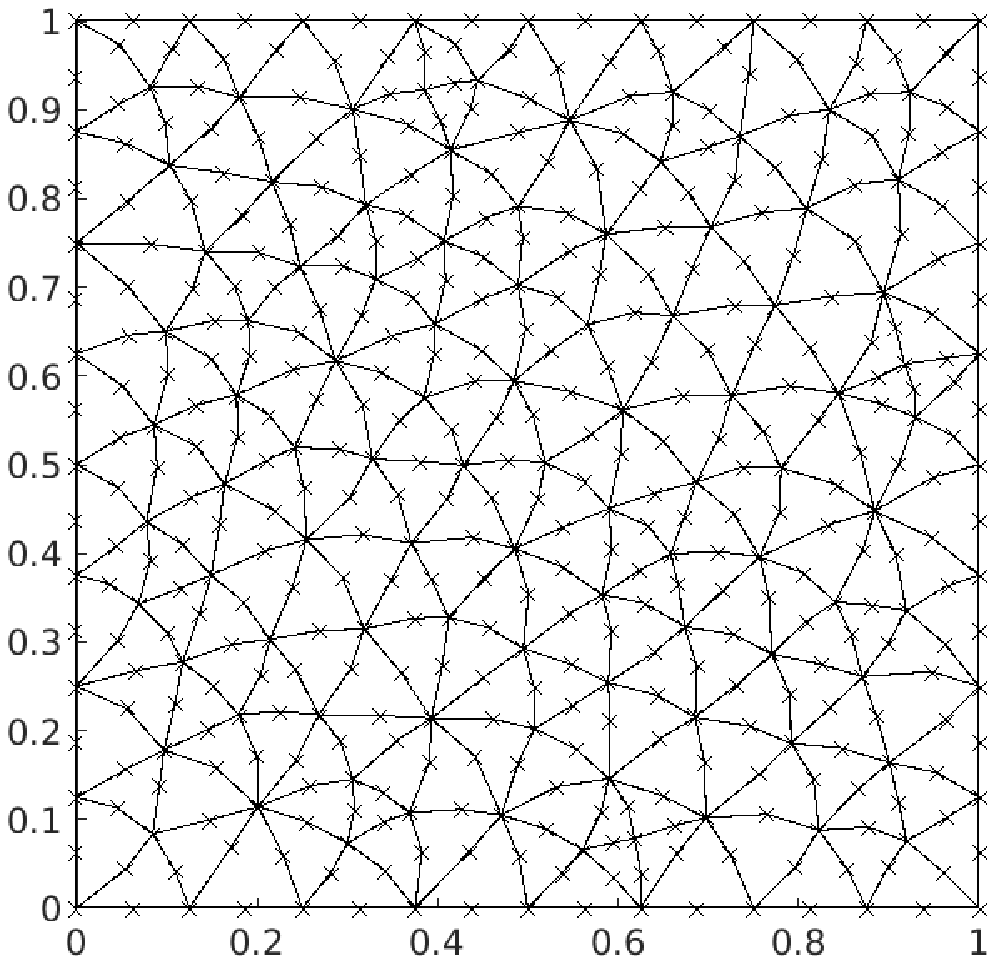}
			\centering\includegraphics[height=5cm, width=6cm]{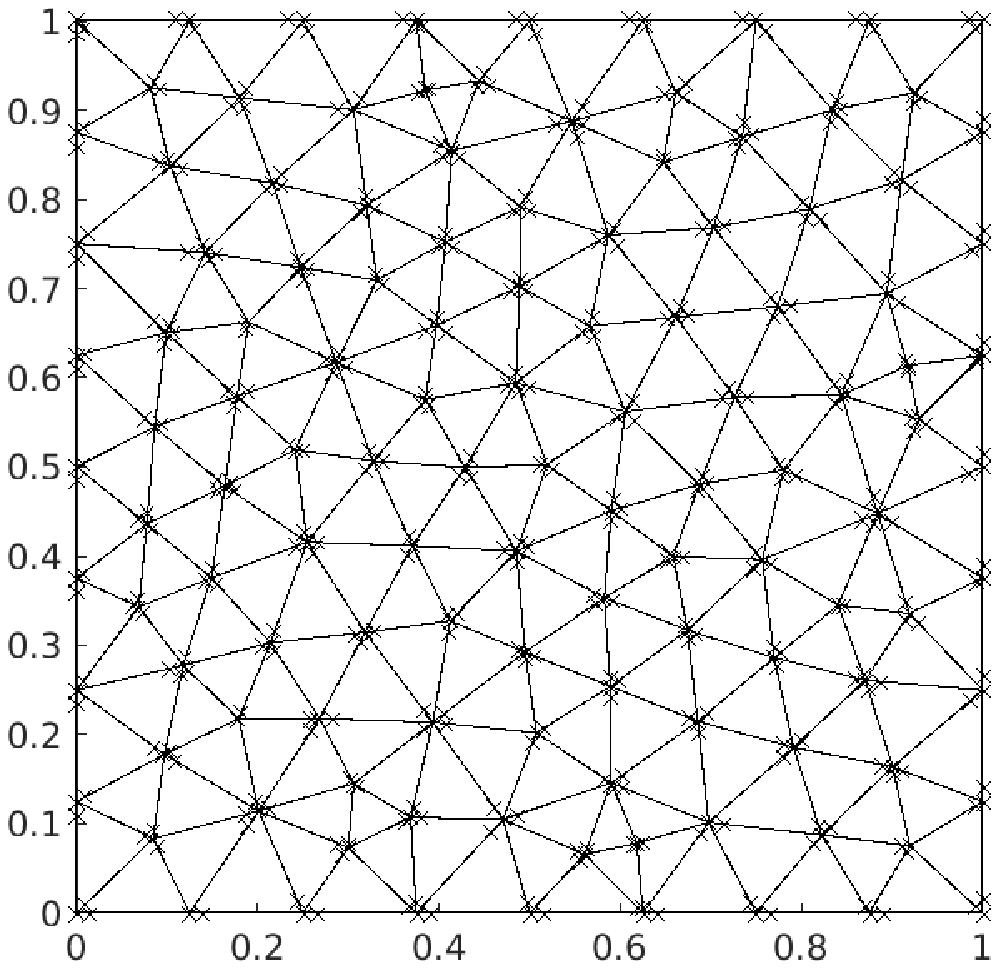}\\
			\centering\includegraphics[height=5cm, width=6cm]{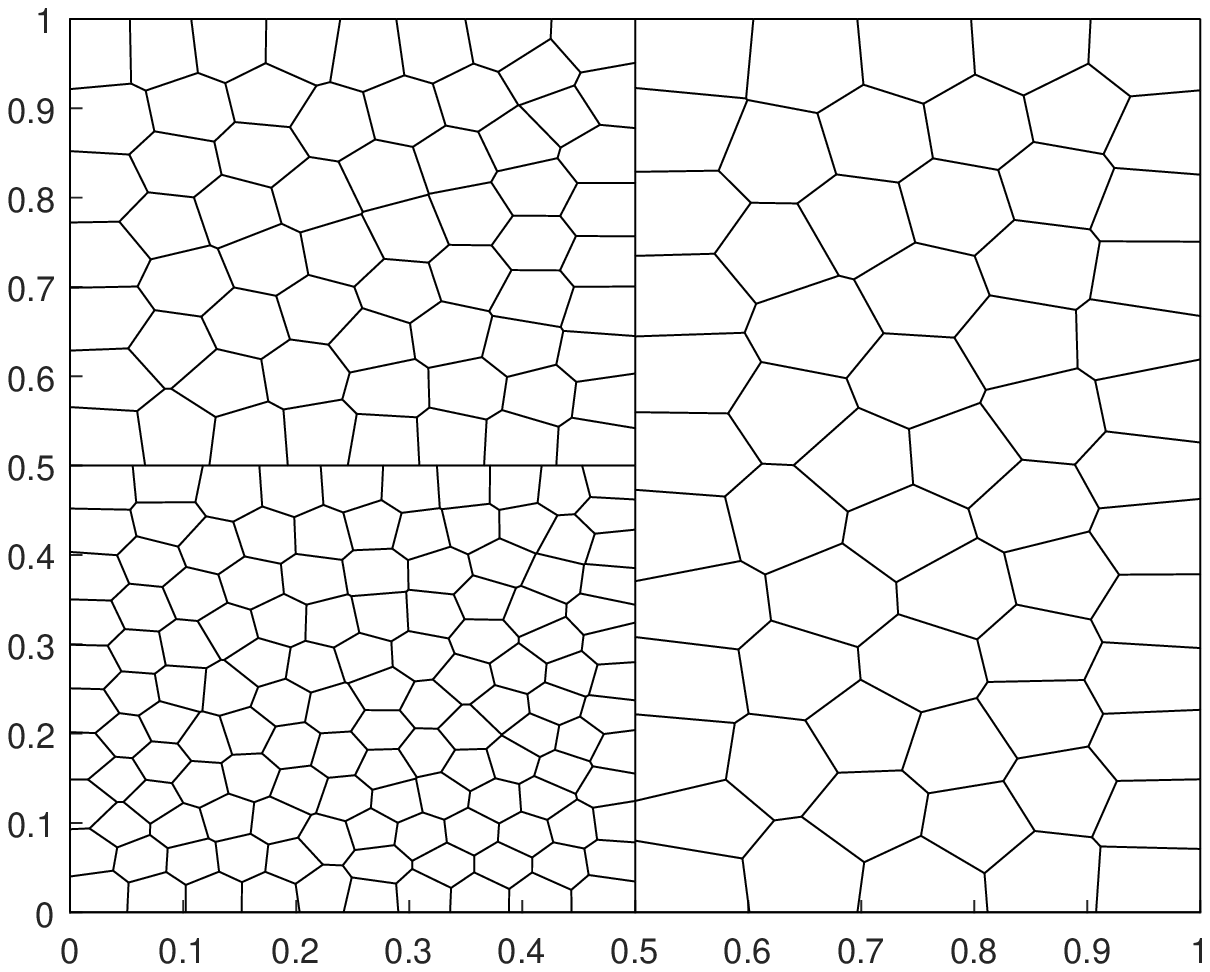}
		\caption{Sample of meshes. Top left: $\mathcal{T}_{h}^{1}$; Top right: $\mathcal{T}_{h}^{2}$; bottom: $\mathcal{T}_{h}^{3}$.
}
		\label{fig:meshesx}
	\end{center}
\end{figure}

Observe that $\mathcal{T}_{h}^{1}$ is such that the middle points allow to consider small edges, whereas $\mathcal{T}_{h}^{2}$ is the standard triangular mesh. The values of the Poisson ration along this   test are  $\nu \in \{0.35,\,0.49\}$.

We have  considered, for simplicity, Young's modulus  $\Lambda = 1$. Also we consider density   $\varrho = 1$. Finally, the stabilization term for this test is 
\begin{equation}\label{classicspec}
S(\textbf{w}_{h},\textbf{v}_{h}) = \alpha\displaystyle{\sum_{E \in \mathcal{T}_{h}} S^{E}(\textbf{w}_{h},\textbf{v}_{h})}, \quad S^{E}(\textbf{w}_{h},\textbf{v}_{h}) = \displaystyle{\sum_{i = 1}^{N_{E}} \textbf{w}_{h}(V_{i})\textbf{v}_{h}(V_{i})},
\end{equation}
where $\alpha := \text{tr}(a_{h}(\cdot,\cdot))/2$. The following tables show approximate values for each of the frequencies $\omega_{i} = \sqrt{\kappa_{i}}$, $i = 1,\ldots,4$, convergence orders and also the extrapolated frequencies, which are adjusted by least-squares by
\begin{equation*}
\omega_{hi} \approx \omega_{i} + C_{i}h^{\alpha_{i}}.
\end{equation*}
We will consider the mesh refinement $\mathrm{N}$ as the number of polygons on the boundary of the square.

\begin{table}[H]
\begin{center}
\begin{tabular}{|c|c|c|c|c|c|c|c|c|} \hline
 $\nu$ & $\omega_{hi}$ & N = 64 & N = 128 & N = 256 & N = 512 & Order & Ext. & \cite{inzunza2021displacementpseudostress} \\ \hline \hline
 \multirow{4}{*}{0.35} & $\omega_{h1}$ & 4.20193 & 4.19522 & 4.19364 & 4.19324 & 2.07 & 4.19313 & 4.19311 \\
 & $\omega_{h2}$ & 4.20261 & 4.19540 & 4.19369 & 4.19325 & 2.06 & 4.19313 & 4.19311 \\
 & $\omega_{h3}$ & 4.39728 & 4.37833 & 4.37373 & 4.37255 & 2.03 & 4.37220 & 4.37217 \\
 & $\omega_{h4}$ & 5.96461 & 5.94118 & 5.93518 & 5.93336 & 1.96 & 5.93309 & 5.93318 \\
\hline
 \multirow{4}{*}{0.49} & $\omega_{h1}$ & 4.32406 & 4.21865 & 4.19634 & 4.19030 & 2.19 & 4.18930 & 4.18858 \\
 & $\omega_{h2}$ & 5.79393 & 5.58095 & 5.53289 & 5.52130 & 2.13 & 5.51817 & 5.51758 \\
 & $\omega_{h3}$ & 5.81843 & 5.58834 & 5.53448 & 5.52161 & 2.09 & 5.51778 & 5.51758 \\
 & $\omega_{h4}$ & 7.08611 & 6.66311 & 6.57261 & 6.55020 & 2.19 & 6.54528 & 6.54337 \\ \hline
\end{tabular}
\end{center}
\caption{Four lowest approximated frequencies, convergence orders, and extrapolated frequencies, computed with  $\mathcal{T}_{h}^{1}$,  $\nu\in\{0.35,\, 0.49\}$, and the stabilization term defined in \eqref{classicspec}}
\label{tabla1}
\end{table}

\begin{table}[H]
\begin{center}
\begin{tabular}{|c|c|c|c|c|c|c|c|c|} \hline
 $\nu$ & $\omega_{hi}$ & N = 64 & N = 128 & N = 256 & N = 512 & Order & Ext. & \cite{inzunza2021displacementpseudostress} \\ \hline \hline
 \multirow{4}{*}{0.35} & $\omega_{h1}$ & 4.20293 & 4.19549 & 4.19371 & 4.19326 & 2.05 & 4.19313 & 4.19311 \\
 & $\omega_{h2}$ & 4.20311 & 4.19553 & 4.19372 & 4.19326 & 2.05 & 4.19313 & 4.19311 \\
 & $\omega_{h3}$ & 4.39907 & 4.37873 & 4.37385 & 4.37258 & 2.04 & 4.37222 & 4.37217 \\
 & $\omega_{h3}$ & 5.96675 & 5.94188 & 5.93535 & 5.93368 & 1.94 & 5.93308 & 5.93318 \\ 
\hline
 \multirow{4}{*}{0.49} & $\omega_{h1}$ & 4.32140 & 4.21722 & 4.19608 & 4.19017 & 2.23 & 4.18936 & 4.18858 \\
 & $\omega_{h2}$ & 5.78751 & 5.57972 & 5.53244 & 5.52112 & 2.13 & 5.51815 & 5.51758 \\
 & $\omega_{h3}$ & 5.80890 & 5.58319 & 5.53294 & 5.52116 & 2.16 & 5.51813 &  5.51758 \\
 & $\omega_{h4}$ & 7.08094 & 6.65902 & 6.57150 & 6.54973 & 2.23 & 6.54557 & 6.54337 \\ \hline
\end{tabular}
\end{center}
\caption{Four lowest approximated frequencies, convergence orders, and extrapolated frequencies, computed with  $\mathcal{T}_{h}^{2}$,  $\nu\in\{0.35,\, 0.49\}$, and the stabilization term defined in \eqref{classicspec}}
\label{tabla2}
\end{table}

\begin{table}[H]
\begin{center}
\begin{tabular}{|c|c|c|c|c|c|c|c|c|} \hline
 $\nu$ & $\omega_{hi}$ & N = 110 & N = 153 & N = 227 & N = 323 & Order & Ext. & \cite{inzunza2021displacementpseudostress} \\ \hline \hline
 \multirow{4}{*}{0.35} & $\omega_{h1}$ & 4.20404 & 4.19810 & 4.19550 & 4.19433 & 2.56 & 4.19376 & 4.19311 \\
 & $\omega_{h2}$ & 4.20423 & 4.19869 & 4.19567 & 4.19438 & 2.13 & 4.19330 & 4.19311 \\
 & $\omega_{h3}$ & 4.40139 & 4.38652 & 4.37913 & 4.37566 & 2.25 & 4.37335 & 4.37217 \\
 & $\omega_{h3}$ & 5.97081 & 5.95194 & 5.94205 & 5.93773 & 2.19 & 5.93441 & 5.93318 \\ 
\hline
 \multirow{4}{*}{0.49} & $\omega_{h1}$ & 4.43393 & 4.31258 & 4.25035 & 4.21936 & 2.13 & 4.19703 & 4.18858 \\
 & $\omega_{h2}$ & 6.00585 & 5.76804 & 5.64825 & 5.58221 & 2.07 & 5.53589 & 5.51758 \\
 & $\omega_{h3}$ & 6.01995 & 5.77483 & 5.64980 & 5.58379 & 2.09 & 5.53682 & 5.51758 \\
 & $\omega_{h4}$ & 7.45588 & 7.01811 & 6.78115 & 6.66172 & 2.00 & 6.56290 & 6.54337 \\ \hline
\end{tabular}
\end{center}
\caption{Four lowest approximated frequencies, convergence orders, and extrapolated frequencies, computed with  $\mathcal{T}_{h}^{3}$,  $\nu\in\{0.35,\, 0.49\}$, and the stabilization term defined in \eqref{classicspec}}
\label{tabla2x}
\end{table}

From Tables \ref{tabla1} and \ref{tabla2} we observe  that the method is capable of compute the frequencies on the square accurately. This is observed from the exotrapolated values that we present, which we compare with those obtained in \cite{inzunza2021displacementpseudostress}  with a mixed finite element method. Also, the  computed frequencies for the both Poisson ratios under consideration converge to the ones on the aforementioned reference independent of the polygonal mesh. In both cases, the quadratic order is attained by the method.

In Figures \ref{fig:plotsspec2x} and \ref{fig:plotsspec2z} we present plots of the first four eigenfunctions, which have been obtained for  $\nu=0.49$ and stabilization term \eqref{classicspec}.
\begin{figure}[H]
	\begin{center}
			\centering\includegraphics[height=6.3cm, width=7.2cm]{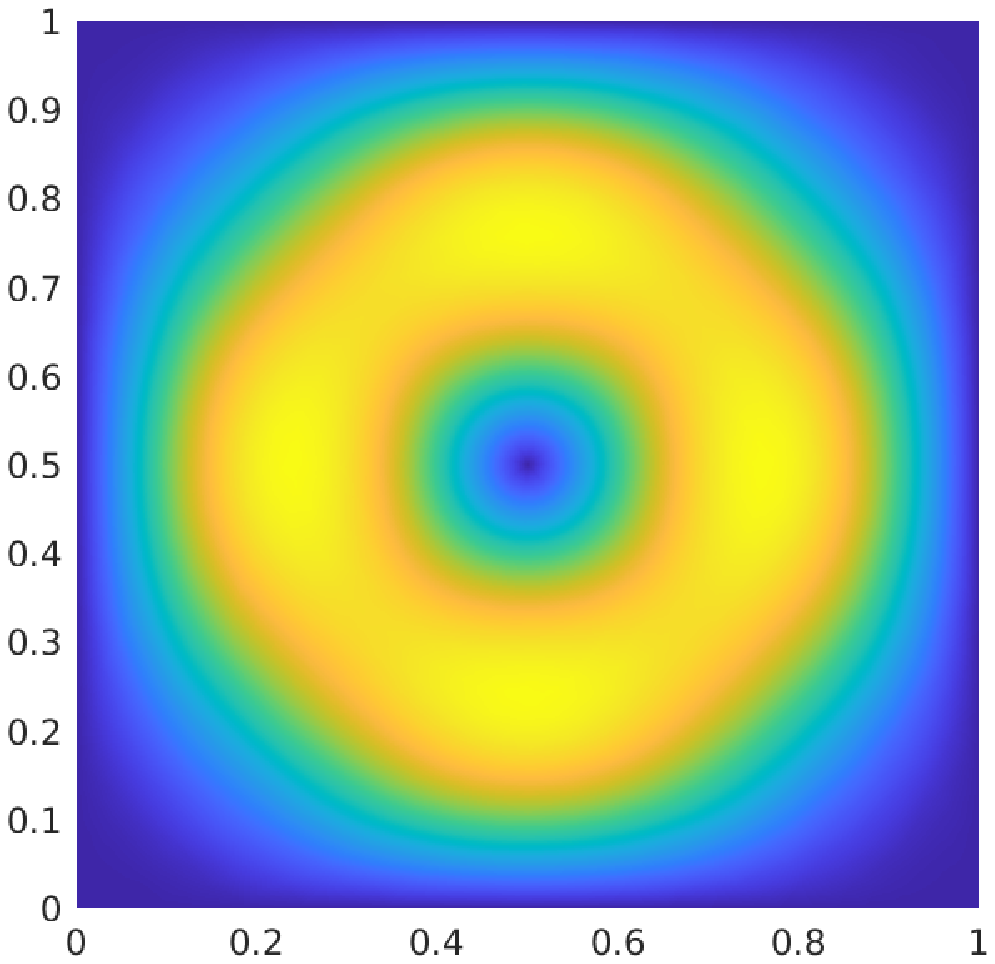}
			\centering\includegraphics[height=6.3cm, width=7.2cm]{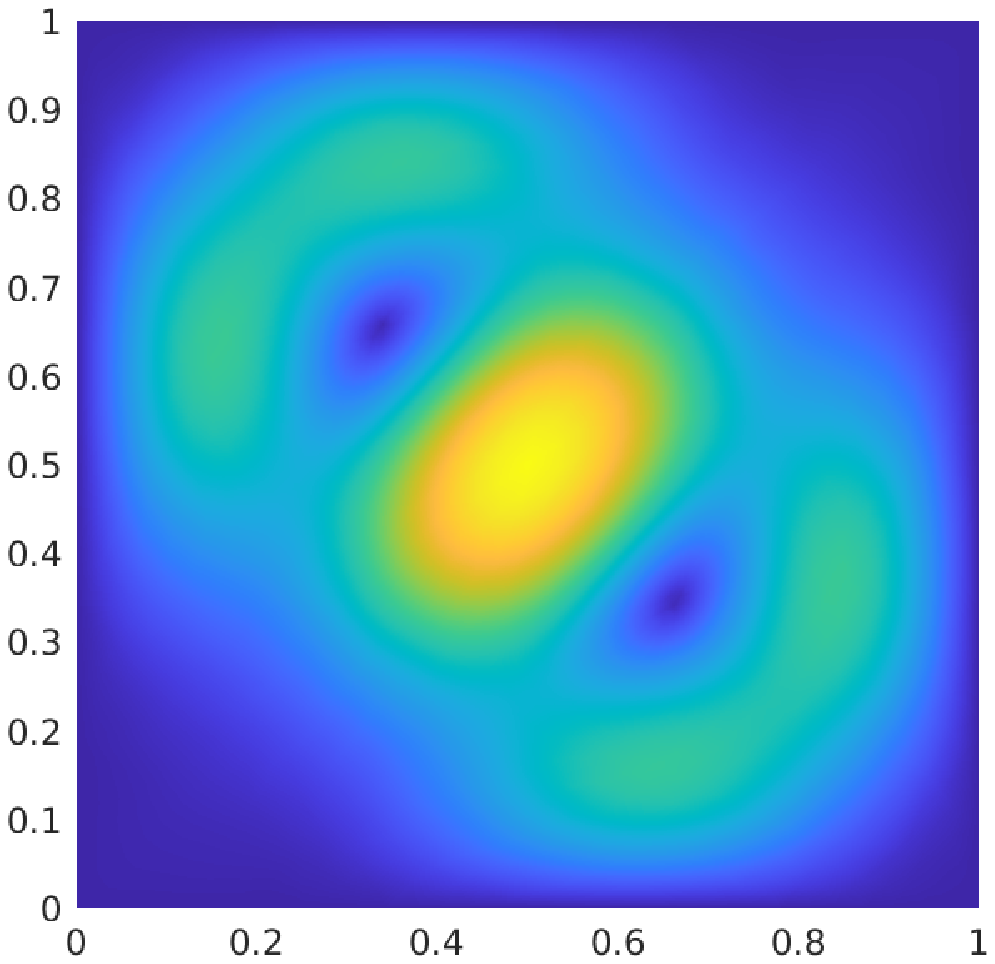}
					\caption{Plots of the first two eigenfunctions computed with $\nu =$ 0.49 and stabilization term \eqref{classicspec}. Left: $\textbf{w}_{h1}$; right: $\textbf{w}_{h2}$.
}
		\label{fig:plotsspec2x}
	\end{center}
\end{figure}
\begin{figure}[H]
\begin{center}
                          \centering\includegraphics[height=6.3cm, width=7.2cm]{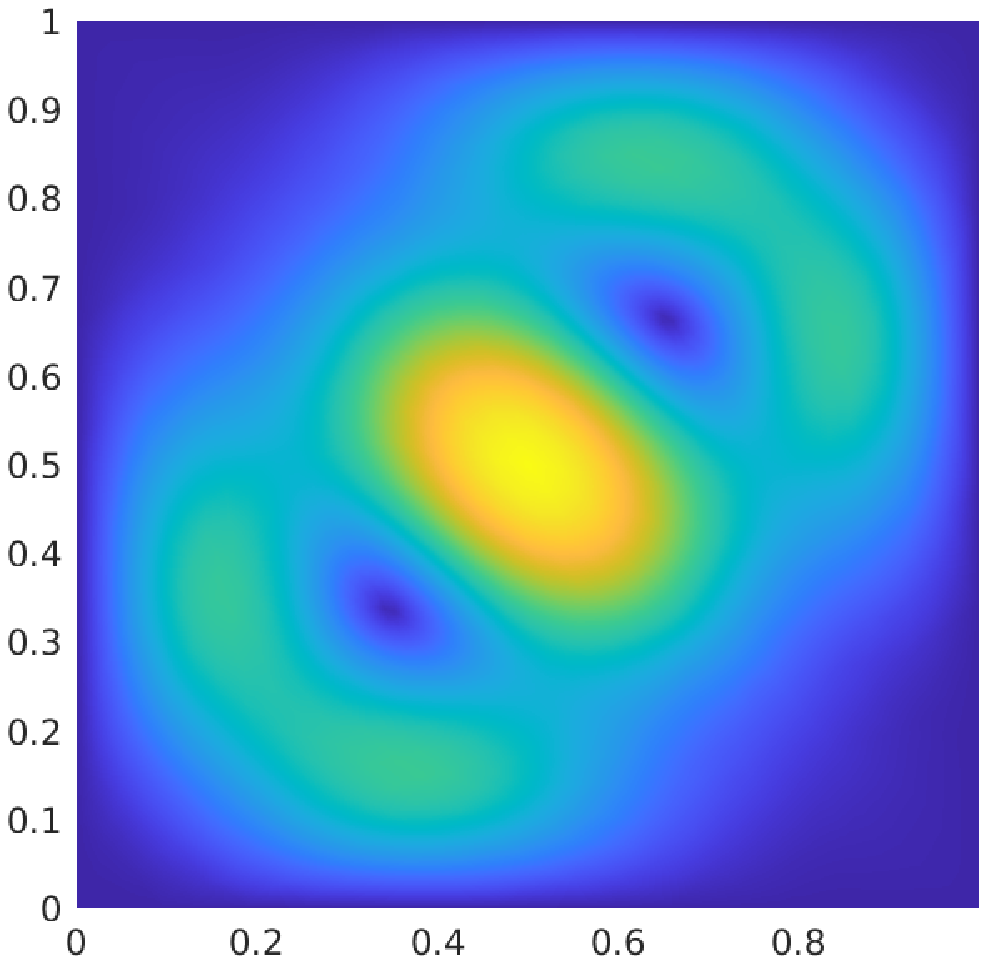}
                          \centering\includegraphics[height=6.3cm, width=7.2cm]{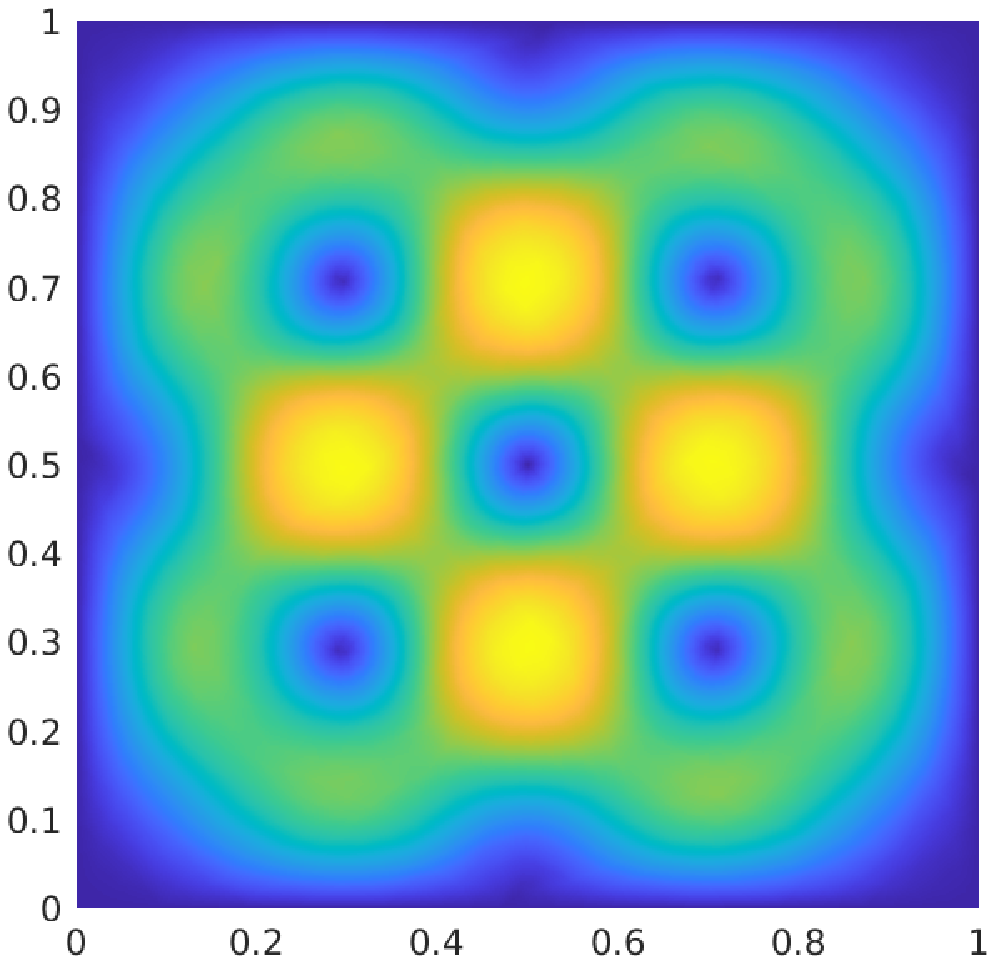}
		\caption{Plots of the third and fourth eigenfunctions computed with  $\nu =$ 0.49 and stabilization term \eqref{classicspec}. Left: $\textbf{w}_{h3}$; right: $\textbf{w}_{h4}$.
}
		\label{fig:plotsspec2z}
	\end{center}
\end{figure}

\subsection{Comparison between the stabilizations}
In order to observe the robustness of the VEM with small edges, we repeat the previous experiments using the following stabilization term
\begin{equation}\label{stabspec}
S(\textbf{w}_{h},\textbf{v}_{h}) = \alpha\displaystyle{\sum_{E \in \mathcal{T}_{h}} S^{E}(\textbf{w}_{h},\textbf{v}_{h})}, \quad S^{E}(\textbf{w}_{h},\textbf{v}_{h}) = h_{E}\displaystyle{\int_{E} \partial_{s}\textbf{w}_{h}\cdot\partial_{s}\textbf{v}_{h}},
\end{equation}
where $\alpha := \text{tr}(a_{h}(\cdot,\cdot))/2$. In the following tables are reported approximated values of each one of the frequencies $\omega_{i} = \sqrt{\kappa_{i}}$, $i = 1,\ldots,4$, convergence orders and extrapolated frequencies which, once again, we compare with the extrapolated ones obtained by \cite{inzunza2021displacementpseudostress} .

\begin{table}[H]
\begin{center}
\begin{tabular}{|c|c|c|c|c|c|c|c|c|} \hline
$\nu$ & $\omega_{hi}$ & N = 64 & N = 128 & N = 256 & N = 512 & Order & Ext. & \cite{inzunza2021displacementpseudostress} \\ \hline \hline
 \multirow{4}{*}{0.35} & $\omega_{h1}$ & 4.20599 & 4.19623 & 4.19390 & 4.19330 & 2.05 & 4.19313 & 4.19311 \\
 & $\omega_{h2}$ & 4.20822 & 4.19680 & 4.19405 & 4.19334 & 2.04 & 4.19313 & 4.19311 \\
 & $\omega_{h3}$ & 4.41110 & 4.38168 & 4.37461 & 4.37276 & 2.04 & 4.37225 & 4.37217 \\
 & $\omega_{h4}$ & 5.98094 & 5.94559 & 5.93631 & 5.93392 & 1.93 & 5.93303 & 5.93318 \\
\hline
 \multirow{4}{*}{0.49} & $\omega_{h1}$ & 4.44484 & 4.24687 & 4.20374 & 4.19202 & 2.15 & 4.18978 & 4.18858 \\
 & $\omega_{h2}$ & 6.04069 & 5.63787 & 5.54703 & 5.52479 & 2.13 & 5.51904 & 5.51758 \\
 & $\omega_{h3}$ & 6.09644 & 5.65738 & 5.55148 & 5.52583 & 2.05 & 5.51768 & 5.51758 \\
 & $\omega_{h4}$ & 7.52201 & 6.76970 & 6.60021 & 6.55686 & 2.12 & 6.54633 & 6.54337 \\ \hline
\end{tabular}
\end{center}
\caption{Four lowest approximated frequencies, convergence orders and extrapolated frequencies, computed with $\mathcal{T}_{h}^{1}$, $\nu\in\{0.35,\,0.49\}$ and stabilization term defined in \eqref{stabspec}.}
\label{tabla3}
\end{table}

\begin{table}[H]
\begin{center}
\begin{tabular}{|c|c|c|c|c|c|c|c|c|} \hline
 $\nu$ & $\omega_{hi}$ & N = 64 & N = 128 & N = 256 & N = 512 & Order & Ext. & \cite{inzunza2021displacementpseudostress} \\ \hline \hline
 \multirow{4}{*}{0.35} & $\omega_{h1}$ & 4.20547 & 4.19609 & 4.19386 & 4.19329 & 2.06 & 4.19313 & 4.19311 \\
&  $\omega_{h2}$ & 4.20577 & 4.19614 & 4.19389 & 4.19330 & 2.07 & 4.19314 & 4.19311 \\
&  $\omega_{h3}$ & 4.40638 & 4.38040 & 4.37430 & 4.37268 & 2.06 & 4.37226 & 4.37217 \\
&  $\omega_{h3}$ & 5.97514 & 5.94417 & 5.93594 & 5.93382 & 1.92 & 5.93302 & 5.93318 \\ 
\hline
 \multirow{4}{*}{0.49} & $\omega_{h1}$ & 4.36861 & 4.22781 & 4.19887 & 4.19074 & 2.22 & 4.18967 & 4.18858 \\
&  $\omega_{h2}$ & 5.88070 & 5.60200 & 5.53775 & 5.52235 & 2.11 & 5.51804 & 5.51758 \\
&  $\omega_{h3}$ & 5.91816 & 5.60777 & 5.53849 & 5.52240 & 2.16 & 5.51821 &  5.51758 \\
&  $\omega_{h4}$ & 7.26788 & 6.70126 & 6.58199 & 6.55196 & 2.21 & 6.54611 & 6.54337 \\ \hline
\end{tabular}
\end{center}
\caption{Four lowest approximated frequencies, convergence orders and extrapolated frequencies, computed with $\mathcal{T}_{h}^{2}$, $\nu\in\{0.35,\,0.49\}$ and stabilization term defined in \eqref{stabspec}.}
\label{tabla4}
\end{table}

\begin{table}[H]
\begin{center}
\begin{tabular}{|c|c|c|c|c|c|c|c|c|} \hline
 $\nu$ & $\omega_{hi}$ & N = 110 & N = 153 & N = 227 & N = 323 & Order & Ext. & \cite{inzunza2021displacementpseudostress} \\ \hline \hline
 \multirow{4}{*}{0.35} & $\omega_{h1}$ & 4.23044 & 4.21115 & 4.20240 & 4.19782 & 2.35 & 4.19539 & 4.19311 \\
&  $\omega_{h2}$ & 4.23264 & 4.21402 & 4.20269 & 4.19790 & 1.91 & 4.19276 & 4.19311 \\
&  $\omega_{h3}$ & 4.48507 & 4.42890 & 4.40000 & 4.38601 & 2.15 & 4.37590 & 4.37217 \\
&  $\omega_{h3}$ & 6.06226 & 6.00188 & 5.96658 & 5.95048 & 1.93 & 5.93472 & 5.93318 \\ 
\hline
 \multirow{4}{*}{0.49} & $\omega_{h1}$ & 5.12641 & 4.68477 & 4.44556 & 4.31929 & 1.96 & 4.21479 & 4.18858 \\
&  $\omega_{h2}$ & 7.23930 & 6.46664 & 6.03793 & 5.78512 & 1.80 & 5.55671 & 5.51758 \\
&  $\omega_{h3}$ & 7.33431 & 6.51547 & 6.05240 & 5.79728 & 1.83 & 5.56155 & 5.51758 \\
&  $\omega_{h4}$ & 9.61583 & 8.26681 & 7.46697 & 7.02478 & 1.73 & 6.56786 & 6.54337 \\ \hline
\end{tabular}
\end{center}
\caption{Four lowest approximated frequencies, convergence orders and extrapolated frequencies, computed with $\mathcal{T}_{h}^{3}$, $\nu\in\{0.35,\,0.49\}$ and stabilization term defined in \eqref{stabspec}.}
\label{tabla4x}
\end{table}
From Tables \ref{tabla3} and \ref{tabla4} we observe that there is no significant differences when the 
stabilization \eqref{classicspec} is changed by \eqref{stabspec}. In fact, the frequencies for the considered Poisson ratios and their extrapolated values are similar. Moreover, the order of convergence is not affected, and the quadratic order is attained perfectly.

\subsection{Nonconvex domain}
The aim of this test is to study the performance of the method in a nonconvex domain. Clearly this geometrical particularity goes beyond from our theoretical assumptions, where the theory is developed on a convex Lispchitz domain. However, computationally we can study the method in order to compare our results with those provided by other numerical  methods.  To do this task, we compute the four smallest frequencies $\omega_{hi}$, $i = 1, \ldots, 4$   
for the L-shaped domain defined by $\Omega := (0,2)^{2}\setminus [1,2)^{2}$. A sample of the meshes  to discretize this domain is presented in Figure \ref{fig:meshesxx}.
\begin{figure}[H]
	\begin{center}
			\centering\includegraphics[height=5cm, width=6cm]{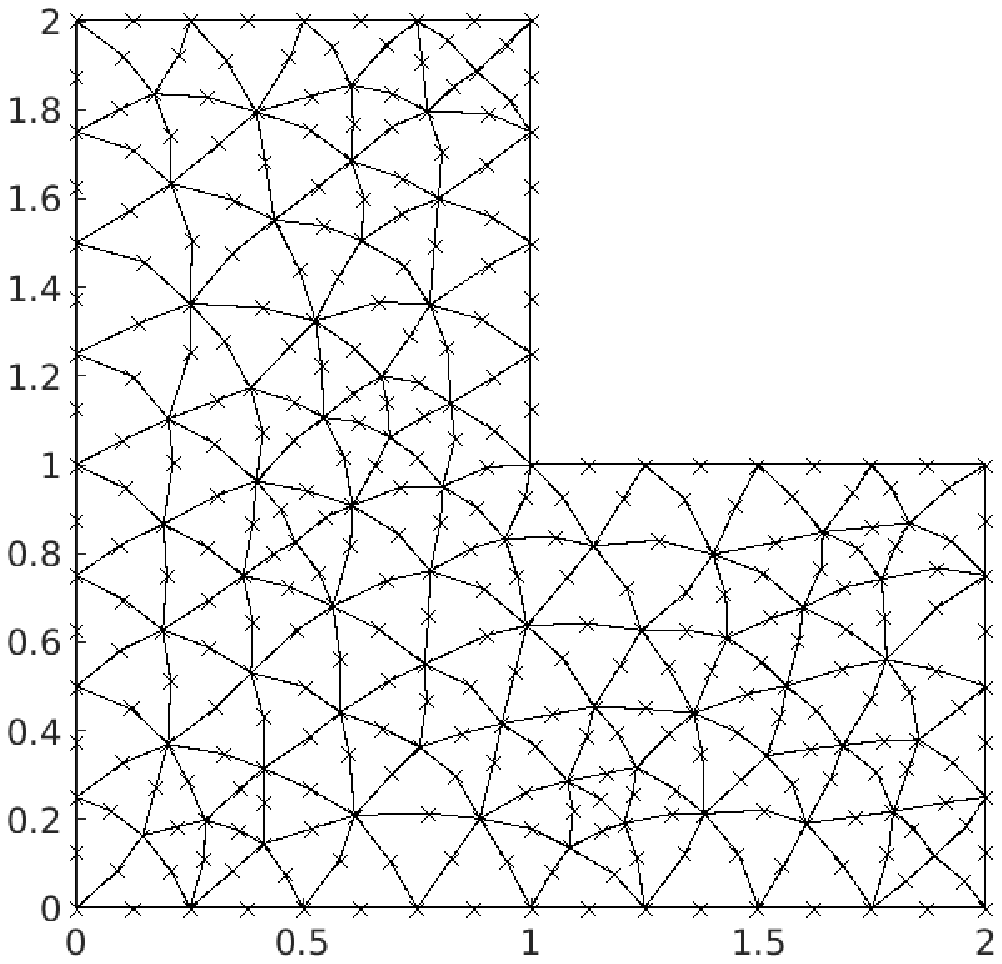}
			\centering\includegraphics[height=5cm, width=6cm]{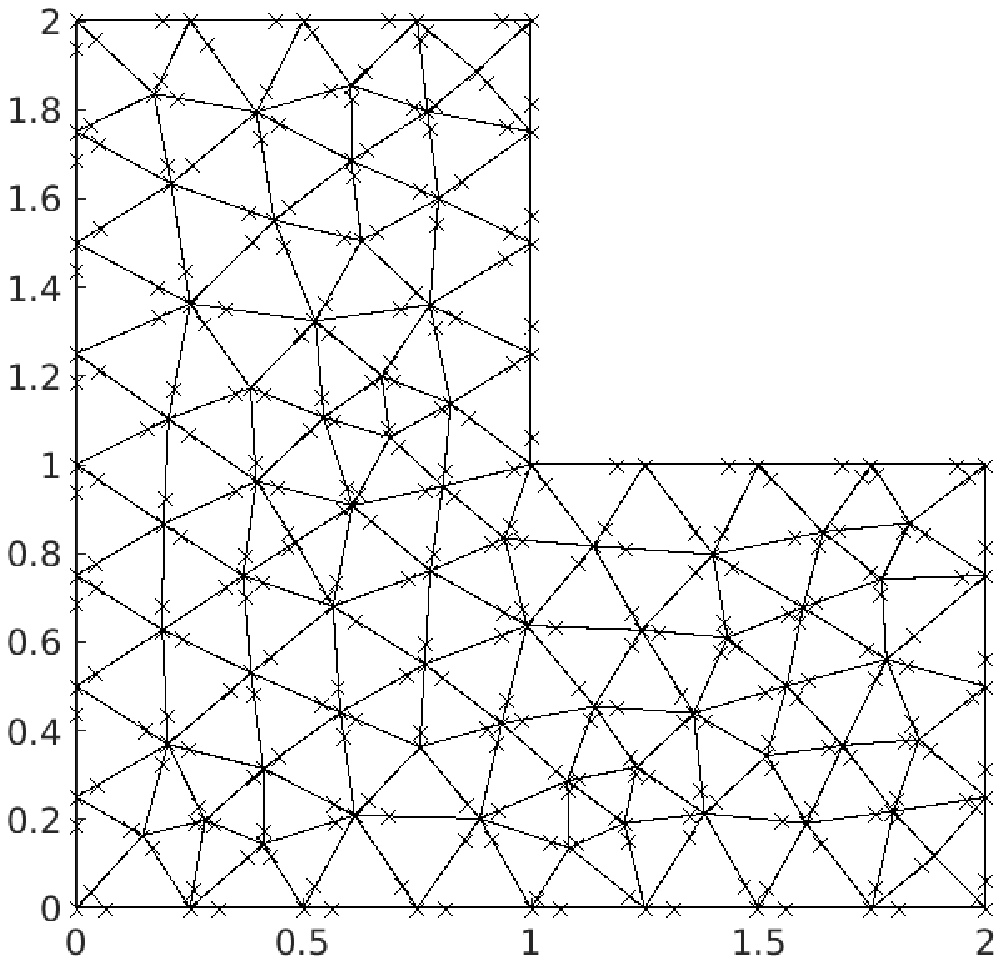}\\
		\caption{Sample of the meshes for the L-shaped domain. Left $\mathcal{T}_{h}^{1}$ (deformed triangles with middle points); right $\mathcal{T}_{h}^{2}$ (triangles with small edges).
}
		\label{fig:meshesxx}
	\end{center}
\end{figure}
In this test we consider the same physical parameters of the previous test, whereas the computed frequencies have been computed with the stabilization term  \eqref{classicspec}, which we scale with the parameter $\alpha := \text{tr}(a_{h}(\cdot,\cdot))/2$. Let us remark that $\mathrm{N}$ represents the number of polygons on the edge of the domain. 

\begin{table}[H]
\begin{center}
\begin{tabular}{|c|c|c|c|c|c|c|c|c|} \hline
$\nu$ & $\omega_{hi}$ & N = 64 & N = 128 & N = 256 & N = 512 & Order & Ext. & \cite{inzunza2021displacementpseudostress} \\ \hline \hline
 \multirow{4}{*}{0.35} & $\omega_{h1}$ & 2.39539 & 2.38512 & 2.38095 & 2.37971 & 1.40 & 2.37871 & 2.37768 \\
 & $\omega_{h2}$ & 2.81163 & 2.80183 & 2.79885 & 2.79805 & 1.75 & 2.79766 & 2.79726 \\
 & $\omega_{h3}$ & 3.33891 & 3.30138 & 3.28635 & 3.28221 & 1.43 & 3.27872 & 3.27876 \\
 & $\omega_{h4}$ & 3.67318 & 3.63581 & 3.62523 & 3.62262 & 1.85 & 3.62140 & 3.62146 \\
\hline
 \multirow{4}{*}{0.49} & $\omega_{h1}$ & 3.60728 & 3.37831 & 3.30437 & 3.28291 & 1.66 & 3.27169 & 3.26734 \\
 & $\omega_{h2}$ & 3.80074 & 3.58525 & 3.52727 & 3.51340 & 1.92 & 3.50750 & 3.50800 \\
 & $\omega_{h3}$ & 4.06885 & 3.80272 & 3.73812 & 3.72280 & 2.05 & 3.71780 & 3.71731 \\
 & $\omega_{h4}$ & 4.52351 & 4.15809 & 4.06992 & 4.04923 & 2.06 & 4.04251 & 4.04256 \\ \hline
\end{tabular}
\end{center}
\caption{Four lowest four computed frequencies, convergence orders and  extrapolated frequencies, computed with  $\mathcal{T}_{h}^{1}$  and the stabilization term \eqref{classicspec}.}
\end{table}

\begin{table}[H]
\begin{center}
\begin{tabular}{|c|c|c|c|c|c|c|c|c|} \hline
$\nu$ & $\omega_{hi}$ & N = 64 & N = 128 & N = 256 & N = 512 & Order & Ext. & \cite{inzunza2021displacementpseudostress} \\ \hline \hline
 \multirow{4}{*}{0.35} & $\omega_{h1}$ & 2.39589 & 2.38554 & 2.38118 & 2.37983 & 1.35 & 2.37870 & 2.37768\\
 & $\omega_{h2}$ & 2.81250 & 2.80218 & 2.79898 & 2.79809 & 1.72 & 2.79765 & 2.79726 \\
 & $\omega_{h3}$ & 3.34109 & 3.30285 & 3.28711 & 3.28261 & 1.40 & 3.27885 & 3.27876\\
 & $\omega_{h4}$ & 3.67592 & 3.63689 & 3.62556 & 3.62272 & 1.82 & 3.62137 & 3.62146\\
\hline
 \multirow{4}{*}{0.49} & $\omega_{h1}$ & 3.57514 & 3.37114 & 3.30296 & 3.28205 & 1.61 & 3.27102 & 3.26734 \\
 & $\omega_{h2}$ & 3.76980 & 3.57930 & 3.52619 & 3.51303 & 1.87 & 3.50720 & 3.50800\\
 & $\omega_{h3}$ & 4.04590 & 3.79714 & 3.73664 & 3.72242 & 2.05 & 3.71771 & 3.71731\\
 & $\omega_{h4}$ & 4.49042 & 4.15356 & 4.06853 & 4.04891 & 2.01 & 4.04165 & 4.04256 \\ \hline
\end{tabular}
\end{center}
\caption{Four lowest computed frequencies, convergence orders and  extrapolated frequencies, computed with $\mathcal{T}_{h}^{2}$ and the stabilization term \eqref{classicspec}.}
\end{table}
Again, we will repeat the previous experiments using the stabilization term \eqref{stabspec}, which will be compared with the results obtained previously. 

\begin{table}[H]
\begin{center}
\begin{tabular}{|c|c|c|c|c|c|c|c|c|} \hline
 $\nu$ & $\omega_{hi}$ & N = 64 & N = 128 & N = 256 & N = 512 & Order & Ext. & \cite{inzunza2021displacementpseudostress} \\ \hline \hline
 \multirow{4}{*}{0.35} & $\omega_{h1}$ & 2.40458 & 2.38790 & 2.38198 & 2.38014 & 1.54 & 2.37905 & 2.37768 \\
 & $\omega_{h2}$ & 2.81958 & 2.80397 & 2.79949 & 2.79823 & 1.80 & 2.79769 & 2.79726 \\
 & $\omega_{h3}$ & 3.36860 & 3.31037 & 3.28963 & 3.28360 & 1.55 & 3.27977 & 3.27876\\
 & $\omega_{h4}$ & 3.70036 & 3.64325 & 3.62715 & 3.62314 & 1.86 & 3.62137 & 3.62146\\
\hline
 \multirow{4}{*}{0.49} & $\omega_{h1}$ & 3.88680 & 3.45600 & 3.33252 & 3.29245 & 1.77 & 3.27835 & 3.26734\\
 & $\omega_{h2}$ & 4.06022 & 3.65189 & 3.54567 & 3.51791 & 1.94 & 3.50811 & 3.50800 \\
 & $\omega_{h3}$ & 4.34083 & 3.87448 & 3.75788 & 3.72789 & 1.99 & 3.71805 & 3.71731 \\
 & $\omega_{h4}$ & 4.71510 & 4.25546 & 4.09518 & 4.05553 & 1.61 & 4.02748 & 4.04256\\ \hline
\end{tabular}
\end{center}
\caption{Four lowest computed frequencies, convergence orders and  extrapolated frequencies, computed with $\mathcal{T}_{h}^{1}$ and the stabilization term \eqref{stabspec}.}
\end{table}

\begin{table}[H]
\begin{center}
\begin{tabular}{|c|c|c|c|c|c|c|c|c|} \hline
 $\nu$ & $\omega_{hi}$ & N = 64 & N = 128 & N = 256 & N = 512 & Order & Ext. & \cite{inzunza2021displacementpseudostress} \\ \hline \hline
 \multirow{4}{*}{0.35} & $\omega_{h1}$ & 2.40061 & 2.38695 & 2.38171 & 2.38004 & 1.44 & 2.37889 & 2.37768\\
 & $\omega_{h2}$ & 2.81689 & 2.80343 & 2.79932 & 2.79819 & 1.74 & 2.79765 & 2.79726 \\
 & $\omega_{h3}$ & 3.35684 & 3.30749 & 3.28878 & 3.28329 & 1.48 & 3.27940 & 3.27876 \\
 & $\omega_{h4}$ & 3.69019 & 3.64083 & 3.62658 & 3.62300 & 1.83 & 3.62135 & 3.62146 \\
\hline
 \multirow{4}{*}{0.49} & $\omega_{h1}$ & 3.66652 & 3.40006 & 3.31136 & 3.28502 & 1.62 & 3.27067 & 3.26734\\
 & $\omega_{h2}$ & 3.85633 & 3.60364 & 3.53214 & 3.51451 & 1.85 & 3.50625 & 3.50800 \\
 & $\omega_{h3}$ & 4.16307 & 3.82425 & 3.74311 & 3.72412 & 2.07 & 3.71807 & 3.71731 \\
 & $\omega_{h4}$ & 4.55870 & 4.19166 & 4.07763 & 4.05117 & 1.76 & 4.03554 & 4.04256 \\ \hline
\end{tabular}
\end{center}
\caption{Four lowest computed frequencies, convergence orders and  extrapolated frequencies, computed with $\mathcal{T}_{h}^{2}$ and the stabilization term \eqref{stabspec}.}
\end{table}

Finally in Figures \ref{fig:plotsspec22x} and \ref{fig:plotsspec22z} we present plots of the first four eigenfunctions obtained for  the L-shaped domain.
\begin{figure}[H]
	\begin{center}
			\centering\includegraphics[height=6.3cm, width=7.2cm]{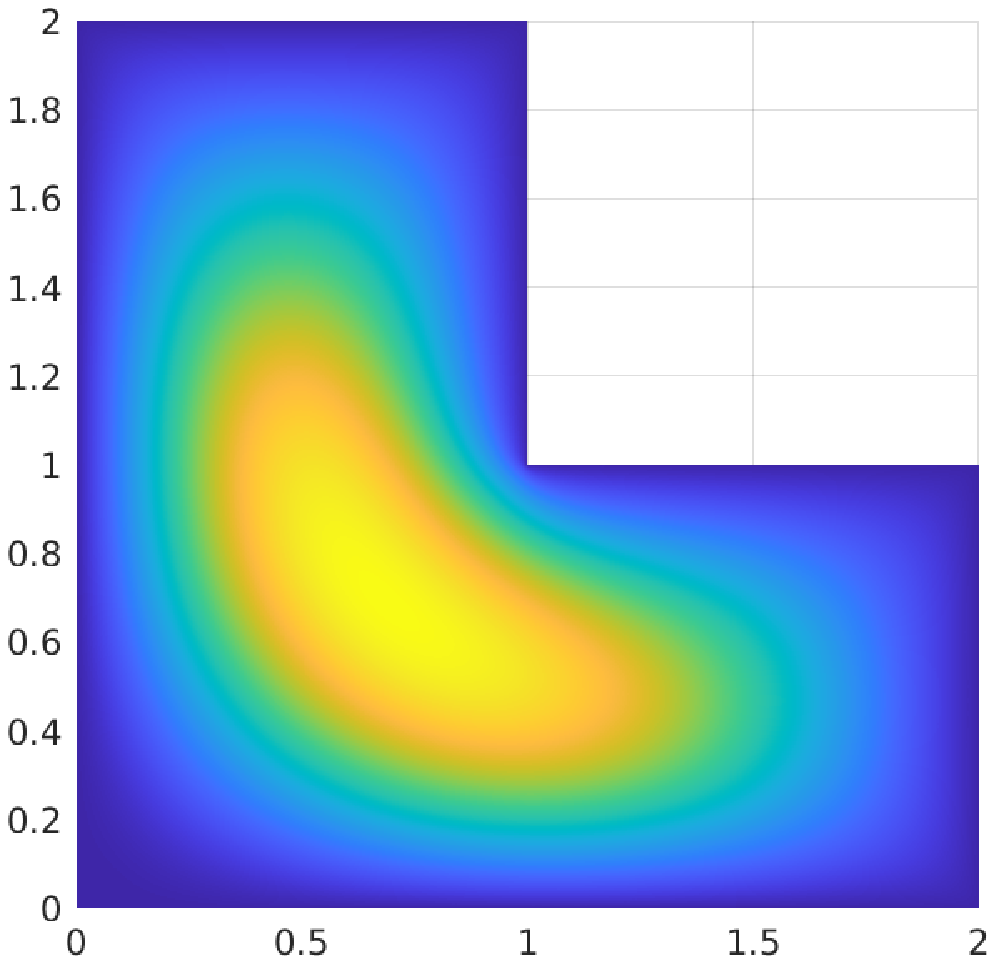}
			\centering\includegraphics[height=6.3cm, width=7.2cm]{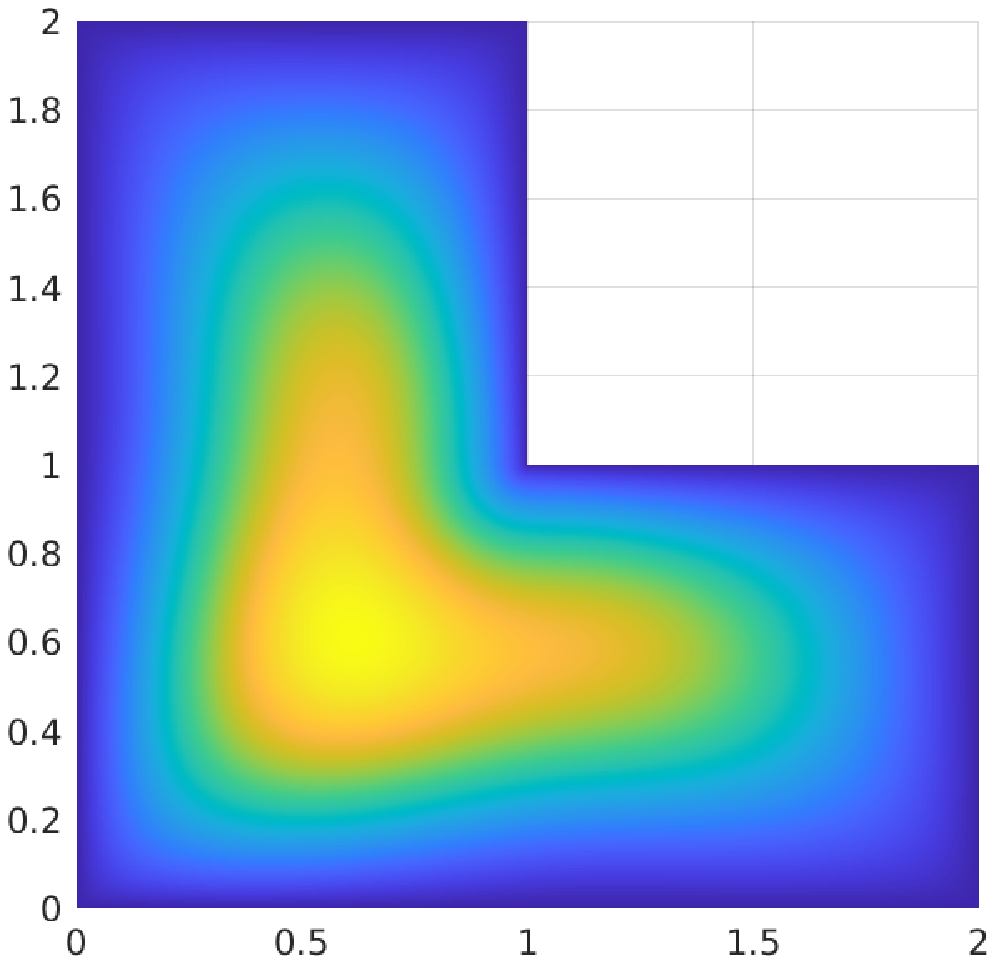}
					\caption{Plots of the first two eigenfunctions computed with  $\nu = 0.35$ and the stabilization term \eqref{stabspec}. Left: $\textbf{w}_{h1}$; right: $\textbf{w}_{h2}$.
}
		\label{fig:plotsspec22x}
	\end{center}
\end{figure}
\begin{figure}[H]
\begin{center}
                          \centering\includegraphics[height=6.3cm, width=7.2cm]{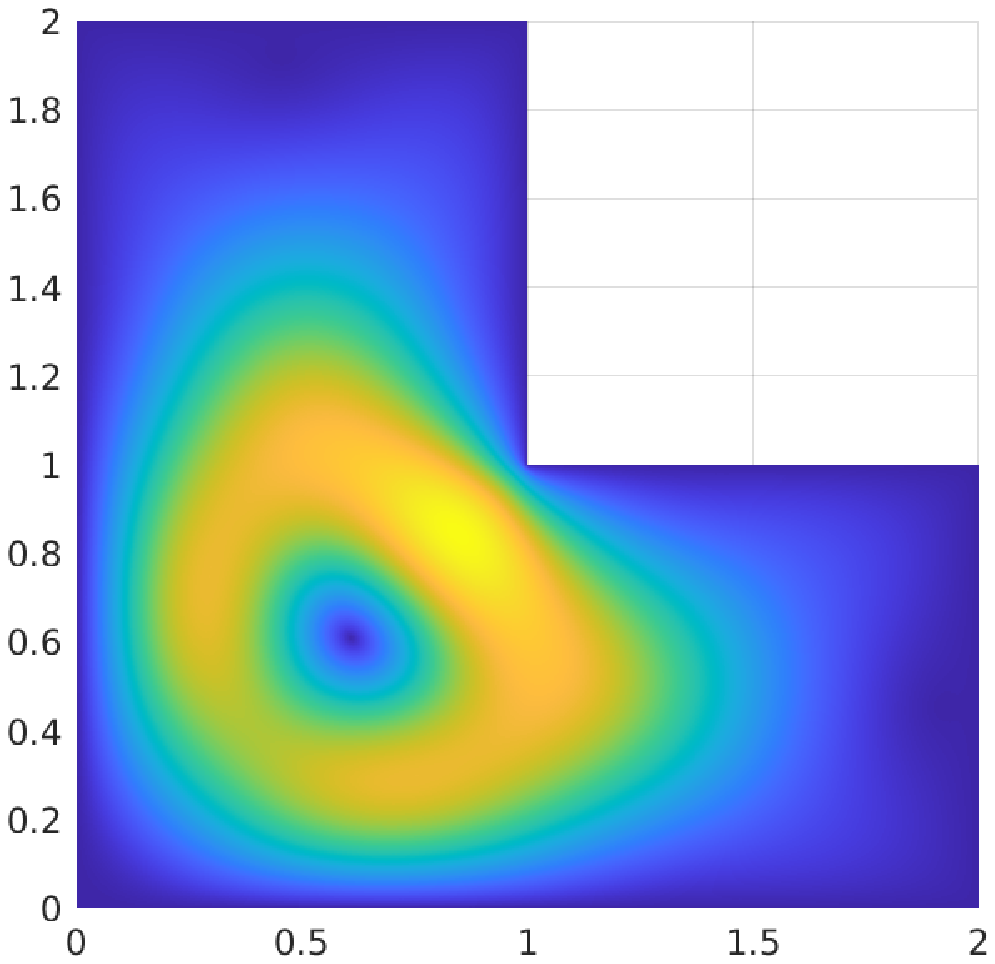}
                          \centering\includegraphics[height=6.3cm, width=7.2cm]{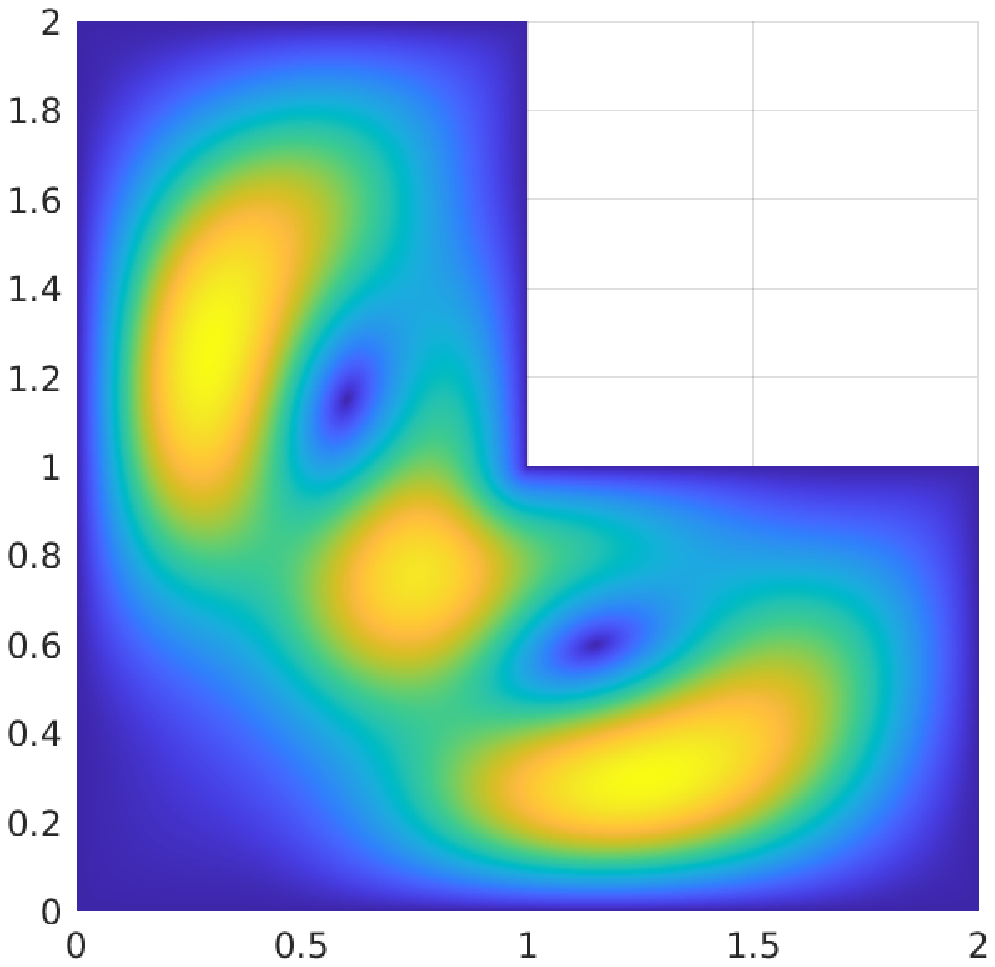}
		\caption{Plots of the third and fourth eigenfunctions computed with  $\nu = 0.35$ and the stabilization term \eqref{stabspec}. Left: $\textbf{w}_{h3}$; right: $\textbf{w}_{h4}$.
}
		\label{fig:plotsspec22z}
	\end{center}
\end{figure}

\subsection{Spurious analysis}

The aim of this test is to analyze the influence of the stabilization parameter of the VEM in the computation of the spectrum. Although the VEM is a robust method to approximate eigenvalues and eigenfunctions, it is well know that the  methods that depend on some parameter may introduce spurious  frequencies. We resort to the reader, for instance, to \cite{MR3962898, MR4229296,MR3340705} for methods that present this nature.

In order to observe more clearly the presence spurious frequencies, we will consider the elasticity spectral problem with mixed boundary conditions. More precisely, the problem of this test reads as follows: Find $\lambda\in\mathbb{R}$ and the displacement  $\textbf{w}$ such that  
\begin{equation}
\label{eq:mixed_boundary_system}
\left\{\begin{array}{cccc}
\textbf{div}(\boldsymbol{\sigma}(\textbf{w})) &=& -\varrho \kappa \textbf{w} \quad \text{in} \; \Omega, \\
\boldsymbol{\sigma}(\textbf{w})\textbf{n} &=&  0 \quad \text{on} \; \Gamma_{N}, \\
\textbf{w} &=& \boldsymbol{0} \quad \text{on} \; \Gamma_{D},
\end{array}\right.
\end{equation}
where  $\Gamma_{D}:= \{(x,0) : x \in (0,1)\}$ and  $\Gamma_{N}$ is the part of the boundary that is not clamped. We need to remark that this problem goes beyond the developed theory, since the regularity for the eigenfunctions under this geometrical configuration is such that $\bw\in\mathbf{H}^{1+s}(\O)$ with $s\in (0,s_{\O})$ and $s_{\O}>0$. Then,  according to \cite{MR3714637}, we need the additional assumption on the mesh:
\begin{itemize}
\item\textbf{A2} There exists $C\in\mathbb{N}$ such that  $N(E)\leq C$, where $N(E)$ represents the number of edges of some polygon $E\in\CT_h$.
\end{itemize}

With this assumption, together with assumption \textbf{A1}, it is possible to perform the analysis but depending on some constant that depends on the size of the mesh. More precisely, according to \cite{MR3714637},  the error estimate has the form 
\begin{equation*}
\|\textbf{u} - \textbf{u}_{h}\|_{1,\Omega} \lesssim \mathfrak{T}(\lambda_{S},\mu_{S})c(h)h^{s-1}|\textbf{u}|_{s,\Omega}, \quad c(h) := c(h) = \underset{E \in \mathcal{T}_{h}}{\max} \; \log\left(1 + \dfrac{h_{E}}{h_{m}(E)}\right),
\end{equation*}
where $h_{m}(E)$ is the smallest edge of the polygon $E$ and $\mathfrak{T}(\lambda_{S},\mu_{S})$ is a positive constant depending on the Lam\'e coefficients. This estimate  is not optimal since the constant $c(h)$ defined above  does not allow  to conclude the convergence in norm between the discrete and continuous solution operators. For this reason, we consider the elasticity problem with mixed boundary conditions only for computational purposes.

To perform the  test, we consider the stabilization term given by \eqref{classicspec} which we rescale  with the parameter  $\beta = 4^{k}$, $-3 \leq k \leq 3$. The meshes  are $\mathcal{T}_{h}^{1}$ and $\mathcal{T}_{h}^{2}$, whereas $\nu \in$ $\{$0.35, 0.45$\}$ and  $\mathrm{N} = 8$.

\begin{table}[H]
\label{tablaespureos1}
\begin{center}
\begin{tabular}{|c||c||c||c||c||c||c||c||c|} \hline
$\beta$ & 1/64 & 1/16 & 1/4 & 1 & 4 & 16 & 64 & $\omega_{i}^{ext.}$ \\ 
\cline{1-9}
$\omega_{h1}$ & 0.6370        & 0.6637         & 0.6769  & 0.6851 & 0.6898 & 0.6916 & 0.7391 & 0.6828 \\
$\omega_{h2}$ & 1.6702        & 1.6877         & 1.6975  & 1.7049 & 1.7095 & 1.7114 & 1.7596 & 1.7015 \\
$\omega_{h3}$ & 1.7519        & 1.7964         & 1.8189  & 1.8341 & 1.8431 & 1.8467 & 2.0362 & 1.8250 \\
$\omega_{h4}$ & 2.7404        & 2.8807         & 2.9388  & 2.9793 & 3.0046 & 3.0148 & 3.4483 & 2.9549 \\
$\omega_{h5}$ & \fbox{2.7954} & 2.9438         & 3.0105  & 3.0512 & 3.0753 & 3.0852 & 3.5296 & 3.0271 \\
$\omega_{h6}$ & \fbox{3.2270} & 3.3851         & 3.4434  & 3.4770 & 3.4979 & 3.5068 & 3.9973 & 3.4503 \\
$\omega_{h7}$ & \fbox{3.4950} & 3.9342         & 4.1230  & 4.2311 & 4.2884 & 4.3101 & 5.2523 & 4.1621 \\
$\omega_{h8}$ & \fbox{4.0069} & 4.4639         & 4.6144  & 4.7122 & 4.7710 & 4.7949 & 5.7256 & 4.6502 \\
$\omega_{h9}$ & \fbox{4.0666} & \fbox{4.5783}  & 4.7610  & 4.8677 & 4.9293 & 4.9539 & 6.2238 & 4.7831 \\
$\omega_{h10}$ & \fbox{4.1824} & \fbox{4.6201} & 4.7926  & 4.9005 & 4.9707 & 5.0011 & 6.2766 & 4.8130 \\ \hline
\end{tabular}
\end{center}
\caption{Computed eigenfrequencies  for  $\mathcal{T}_{h}^{1}$, $\mathrm{N} = 8$ and  $\nu = 0.35$.}
\end{table}

\begin{table}[H]
\label{tablaespureos3}
\begin{center}
\begin{tabular}{|c||c||c||c||c||c||c||c||c|} \hline
$\beta$ & 1/64 & 1/16 & 1/4 & 1 & 4 & 16 & 64 & $\omega_{i}^{ex}$ \\ 
\cline{1-9}
$\omega_{h1}$  & 0.6793 & 0.6812 & 0.6849 & 0.6972 & 0.7280 & 0.7912 & 0.8782 & 0.6828 \\
$\omega_{h2}$  & 1.7007 & 1.7020 & 1.7058 & 1.7170 & 1.7423 & 1.7908 & 1.8625 & 1.7015 \\
$\omega_{h3}$  & 1.8300 & 1.8337 & 1.8424 & 1.8686 & 1.9466 & 2.1412 & 2.5942 & 1.8250 \\
$\omega_{h4}$  & 2.9436 & 2.9478 & 2.9812 & 3.0773 & 3.2096 & 3.3744 & 3.8843 & 2.9549 \\
$\omega_{h5}$  & 3.0600 & 3.0703 & 3.0819 & 3.1232 & 3.3513 & 3.7429 & 4.3502 & 3.0271 \\
$\omega_{h6}$  & 3.4685 & 3.4745 & 3.5023 & 3.5915 & 3.8299 & 4.5374 & \fbox{5.1987} & 3.4503 \\
$\omega_{h7}$  & 4.2649 & 4.2668 & 4.3314 & 4.4319 & 4.6500 & 4.9312 & \fbox{6.2401} & 4.1621 \\
$\omega_{h8}$  & 4.6596 & 4.6742 & 4.7440 & 4.9272 & 5.2994 & 6.0221 & \fbox{6.5812} & 4.6502 \\
$\omega_{h9}$  & 4.8101 & 4.8400 & 4.9257 & 5.1326 & 5.7235 & \fbox{6.2680} & \fbox{7.6194} & 4.7831 \\
$\omega_{h10}$ & 4.8684 & 4.8868 & 4.9392 & 5.2301 & 5.7377 & \fbox{6.5411} & \fbox{8.5473} & 4.8130 \\ \hline
\end{tabular}
\end{center}
\caption{Computed eigenfrequencies for $\mathcal{T}_{h}^{2}$, $\mathrm{N}= 8$ and  $\nu = 0.35$.}
\end{table}

\begin{table}[H]
\label{tablaespureos2}
\begin{center}
\begin{tabular}{|c||c||c||c||c||c||c||c||c|} \hline
$\beta$ & 1/64 & 1/16 & 1/4 & 1 & 4 & 16 & 64 & $\omega_{i}^{ex}$ \\ 
\cline{1-9}
$\omega_{h1}$ & 0.6434         & 0.6725 & 0.6880 & 0.6989 & 0.7069 & 0.7119 & 0.7138 & 0.6967 \\
$\omega_{h2}$ & 1.7398         & 1.7709 & 1.7897 & 1.8055 & 1.8186 & 1.8268 & 1.8300 & 1.7996 \\
$\omega_{h3}$ & 1.7759         & 1.8194 & 1.8414 & 1.8572 & 1.8698 & 1.8776 & 1.8806 & 1.8481\\
$\omega_{h4}$ & 2.7400         & 2.8798 & 2.9433 & 2.9852 & 3.0182 & 3.0401 & 3.0493 & 2.9630\\
$\omega_{h5}$ & 2.8094         & 2.9497 & 3.0039 & 3.0425 & 3.0732 & 3.0924 & 3.1001 & 3.0212\\
$\omega_{h6}$ & \fbox{3.2658}  & 3.4747 & 3.5656 & 3.6282 & 3.6817 & 3.7186 & 3.7342 & 3.5849\\
$\omega_{h7}$ & \fbox{3.4872}  & 3.9152 & 4.0972 & 4.2019 & 4.2715 & 4.3122 & 4.3282 & 4.1386\\
$\omega_{h8}$ & \fbox{4.0111}  & 4.5186 & 4.7094 & 4.8230 & 4.9067 & 4.9571 & 4.9766 & 4.7379\\
$\omega_{h9}$ & \fbox{4.0763}  & \fbox{4.5663} & 4.7246 & 4.8361 & 4.9408 & 5.0152 & 5.0462 & 4.7485\\
$\omega_{h10}$ & \fbox{4.2360} & 4.8963 & 5.1495 & 5.3143 & 5.4419 & 5.5252 & 5.5594 & 5.1977 \\ \hline
\end{tabular}
\end{center}
\caption{Computed eigenfrequencies for $\mathcal{T}_{h}^{1}$, $\mathrm{N} = 8$ and  $\nu = 0.45$.}
\end{table}

\begin{table}[H]
\label{tablaespureos4}
\begin{center}
\begin{tabular}{|c||c||c||c||c||c||c||c||c|} \hline
$\beta$ & 1/64 & 1/16 & 1/4 & 1 & 4 & 16 & 64 & $\omega_{i}^{ex}$ \\ 
\cline{1-9}
$\omega_{h1}$  & 0.6910 & 0.6929 & 0.7009 & 0.7142 & 0.7459 & 0.8025 & 0.8671 & 0.6967 \\
$\omega_{h2}$  & 1.7917 & 1.7955 & 1.8078 & 1.8319 & 1.8937 & 2.0212 & 2.2613 & 1.7996 \\
$\omega_{h3}$  & 1.8521 & 1.8547 & 1.8644 & 1.8938 & 1.9765 & 2.1973 & 2.5647 & 1.8481 \\
$\omega_{h4}$  & 2.9895 & 2.9931 & 3.0152 & 3.0608 & 3.1659 & 3.3688 & 3.9960 & 2.9630 \\
$\omega_{h5}$  & 3.0092 & 3.0218 & 3.0497 & 3.1410 & 3.4073 & 4.1418 & 5.0916 & 3.0212 \\
$\omega_{h6}$  & 3.6046 & 3.6167 & 3.6642 & 3.8176 & 4.1364 & 4.7753 & \fbox{5.3750} & 3.5849 \\
$\omega_{h7}$  & 4.2547 & 4.2695 & 4.3106 & 4.4130 & 4.5875 & 4.9213 & \fbox{6.4913} & 4.1386 \\
$\omega_{h8}$  & 4.7347 & 4.7525 & 4.8695 & 5.1082 & 5.7296 & 6.9605 & \fbox{8.5122} & 4.7379 \\
$\omega_{h9}$  & 4.7961 & 4.8302 & 4.8863 & 5.1754 & 5.7952 & \fbox{7.0476} & \fbox{9.3058} & 4.7485 \\
$\omega_{h10}$ & 5.4019 & 5.4274 & 5.4936 & 5.6735 & 6.1776 & \fbox{7.2649} & \fbox{9.4293} & 5.1977 \\ \hline
\end{tabular}
\end{center}
\caption{Computed eigenfrequencies with $\mathcal{T}_{h}^{2}$, $\mathrm{N} = 8$ and  $\nu = 0.45$.}
\end{table}

In the case of $\mathcal{T}_{h}^{1}$ mesh, note that the spurious appears when $\beta = 1/64$, while in the case of $\mathcal{T}_{h}^{2}$ mesh, spurious appears when $\beta = 64$. In the following tables are report for each mesh, approximated frequencies for each refinement, with the aim of analyzing the presence of spurious. We denote $\mathrm{N}$ the number of polygons in one side of the square. 

\begin{table}[H]
\begin{center}
\begin{tabular}{|c|c|c|c|c|c|} \hline
$\nu$ & $\omega_{hi}$ & $\mathrm{N} = 8$ & $\mathrm{N} = 16$ & $\mathrm{N} = 32$ & $\mathrm{N} = 64$ \\ \hline \hline 
\multirow{10}{*}{0.35} & $\omega_{h1}$ & 0.6370 & 0.6669 & 0.6759 & 0.6791 \\
 & $\omega_{h2}$ & 1.6702 & 1.6877 & 1.6940 & 1.6974 \\
 & $\omega_{h3}$ & 1.7519 & 1.7990 & 1.8159 & 1.8208 \\
 & $\omega_{h4}$ & 2.7404 & 2.8854 & 2.9332 & 2.9442 \\
 & $\omega_{h5}$ & \fbox{2.7954} & 2.9472 & 2.9972 & 3.0118 \\
 & $\omega_{h6}$ & \fbox{3.2270} & 3.3922 & 3.4292 & 3.4399 \\
 & $\omega_{h7}$ & \fbox{3.4950} & \fbox{3.9227} & 4.1007 & 4.1301 \\
 & $\omega_{h8}$ & \fbox{4.0069} & \fbox{4.4857} & 4.5878 & 4.6207 \\
 & $\omega_{h9}$ & \fbox{4.0666} & \fbox{4.5951} & 4.7167 & 4.7511 \\
 & $\omega_{h10}$ & \fbox{4.1824} & \fbox{4.6364} & 4.7458 & 4.7763 \\ \hline
 \hline 
\multirow{10}{*}{0.45} & $\omega_{h1}$ & 0.6434 & 0.6747 & 0.6854 & 0.6895 \\
& $\omega_{h2}$ & 1.7398 & 1.7700 & 1.7823 & 1.7889 \\
& $\omega_{h3}$ & 1.7759 & 1.8221 & 1.8385 & 1.8433 \\
& $\omega_{h4}$ & 2.7400 & 2.8813 & 2.9278 & 2.9424 \\
& $\omega_{h5}$ & 2.8094 & 2.9515 & 2.9986 & 3.0094 \\
& $\omega_{h6}$ & \fbox{3.2658} & 3.4786 & 3.5398 & 3.5597 \\
& $\omega_{h7}$ & \fbox{3.4872} & \fbox{3.9011} & 4.0738 & 4.1024 \\
& $\omega_{h8}$ & \fbox{4.0111} & \fbox{4.5386} & 4.6616 & 4.6990 \\ 
& $\omega_{h9}$ & \fbox{4.0763} & \fbox{4.5718} & 4.6723 & 4.7012 \\
& $\omega_{h10}$ & \fbox{4.2360} & \fbox{4.9083} & 5.0910 & 5.1417 \\ \hline
\end{tabular}
\end{center}
\caption{First ten approximated frequencies  for  $\mathcal{T}_{h}^{1}$ and  $\beta = 1/64$.}
\end{table}

\begin{table}[H]
\begin{center}
\begin{tabular}{|c|c|c|c|c|c|} \hline
$\nu$ & $\omega_{hi}$ & $\mathrm{N} = 8$ & $\mathrm{N} = 16$ & $\mathrm{N} = 32$ & $\mathrm{N} = 64$ \\ \hline \hline 
\multirow{10}{*}{0.35} & $\omega_{h1}$ & 0.8782 & 0.7867 & 0.7247 & 0.6966 \\
 & $\omega_{h2}$ & 1.8625 & 1.7847 & 1.7376 & 1.7146 \\
 & $\omega_{h3}$ & 2.5942 & 2.1211 & 1.9213 & 1.8524 \\
 & $\omega_{h4}$ & 3.8843 & 3.3080 & 3.1449 & 3.0641 \\
 & $\omega_{h5}$ & 4.3502 & 3.7277 & 3.3147 & 3.0691 \\
 & $\omega_{h6}$ & \fbox{5.1987} & \fbox{4.4298} & 3.7518 & 3.5443 \\
 & $\omega_{h7}$ & \fbox{6.2401} & \fbox{4.7142} & 4.4244 & 4.2589 \\
 & $\omega_{h8}$ & \fbox{6.5812} & 5.9087 & 5.1695 & 4.8444 \\
 & $\omega_{h9}$ & \fbox{7.6194} & 6.1201 & 5.5569 & 4.9985 \\
 & $\omega_{h10}$ & \fbox{8.5473} & \fbox{6.4273} & 5.5702 & 5.1205 \\ \hline
 \hline 
\multirow{10}{*}{0.45} & $\omega_{h1}$ & 0.8671 & 0.8004 & 0.7436 & 0.7138 \\
& $\omega_{h2}$ & 2.2613 & 2.0133 & 1.8837 & 1.8285 \\
& $\omega_{h3}$ & 2.5647 & 2.1753 & 1.9564 & 1.8777 \\
& $\omega_{h4}$ & 3.9960 & 3.3243 & 3.0973 & 3.0092 \\
& $\omega_{h5}$ & 5.0916 & 4.1074 & 3.3905 & 3.1247 \\
& $\omega_{h6}$ & \fbox{5.3750} & \fbox{4.6784} & 4.0354 & 3.7370 \\
& $\omega_{h7}$ & \fbox{6.4913} & \fbox{4.7309} & 4.3834 & 4.2245 \\
& $\omega_{h8}$ & \fbox{8.5122} & 6.6766 & 5.5418 & 4.9659 \\ 
& $\omega_{h9}$ & \fbox{9.3058} & 6.9063 & 5.6267 & 5.0895 \\
& $\omega_{h10}$ & \fbox{9.4293} & \fbox{6.9771} & 5.8199 & 5.3521 \\ \hline
\end{tabular}
\end{center}
\caption{First ten approximated frequencies  for  $\mathcal{T}_{h}^{2}$ and  $\beta = 64$.}
\end{table}

Finally, for each stabilization parameter $\beta$, we perform the  analysis of convergence orders using $\mathcal{T}_{h}^{1}$ mesh, with $\nu=0.45$. Ee remark that the results for other Poisson ratios are similar, and hence we do not include it. In the following tables we report approximated frequencies, convergence orders and extrapolated frequencies. Once again, $\mathrm{N}$ denotes the number of polygons in one side of the square. 

\begin{table}[H]
\begin{center}
\begin{tabular}{|c|c|c|c|c|c|c|c|} \hline
$\beta$ & $\omega_{hi}$ & $\mathrm{N} = 8$ & $\mathrm{N} = 16$ & $\mathrm{N} = 32$ & $\mathrm{N} = 64$ & Orden & Extrap. \\ \hline \hline
\multirow{4}{*}{$\dfrac{1}{64}$} & $\omega_{h1}$ & 0.6434 & 0.6747 & 0.6854 & 0.6895 & 1.51 & 0.6915 \\
& $\omega_{h2}$ & 1.7398 & 1.7700 & 1.7823 & 1.7890 & 1.18 & 1.7933 \\ 
& $\omega_{h3}$ & 1.7759 & 1.8221 & 1.8385 & 1.8433 & 1.55 & 1.8463 \\
& $\omega_{h4}$ & 2.7400 & 2.8813 & 2.9278 & 2.9424 & 1.62 & 2.9497 \\ \hline \hline
\multirow{4}{*}{$\dfrac{1}{16}$} & $\omega_{h1}$ & 0.6725 & 0.6850 & 0.6894 & 0.6911 & 1.47 & 0.6920 \\ 
& $\omega_{h2}$ & 1.7709 & 1.7832 & 1.7886 & 1.7914 & 1.13 & 1.7935 \\
& $\omega_{h3}$ & 1.8194 & 1.8368 & 1.8426 & 1.8442 & 1.64 & 1.8451 \\
& $\omega_{h4}$ & 2.8798 & 2.9256 & 2.9416 & 2.9469 & 1.54 & 2.9498 \\ \hline \hline
\multirow{4}{*}{$\dfrac{1}{4}$} & $\omega_{h1}$ & 0.6880 & 0.6908 & 0.6918 & 0.6921 & 1.56 & 0.6923 \\
& $\omega_{h2}$ & 1.7897 & 1.7913 & 1.7923 & 1.7928 & 0.79 & 1.7936 \\
& $\omega_{h3}$ & 1.8414 & 1.8435 & 1.8444 & 1.8446 & 1.40 & 1.8448 \\
& $\omega_{h4}$ & 2.9433 & 2.9469 & 2.9486 & 2.9493 & 1.14 & 2.9499 \\ \hline \hline
\multirow{4}{*}{$1$} & $\omega_{h1}$ & 0.6989 & 0.6952 & 0.6936 & 0.6929 & 1.18 & 0.6923 \\
& $\omega_{h2}$ & 1.8055 & 1.7977 & 1.7950 & 1.7939 & 1.46 & 1.7933 \\
& $\omega_{h3}$ & 1.8572 & 1.8478 & 1.8454 & 1.8449 & 2.01 & 1.8447 \\
& $\omega_{h4}$ & 2.9852 & 2.9607 & 2.9532 & 2.9509 & 1.70 & 2.9499 \\ \hline \hline
\multirow{4}{*}{$4$} & $\omega_{h1}$ & 0.7069 & 0.6987 & 0.6950 & 0.6934 & 1.17 & 0.6921 \\
& $\omega_{h2}$ & 1.8186 & 1.8032 & 1.7970 & 1.7947 & 1.35 & 1.7931 \\ 
& $\omega_{h3}$ & 1.8698 & 1.8512 & 1.8463 & 1.8451 & 1.93 & 1.8446 \\
& $\omega_{h4}$ & 3.0182 & 2.9720 & 2.9569 & 2.9522 & 1.63 & 2.9499 \\ \hline \hline
\multirow{4}{*}{$16$} & $\omega_{h1}$ & 0.7119 & 0.7010 & 0.6959 & 0.6938 & 1.14 & 0.6918 \\
& $\omega_{h2}$ & 1.8268 & 1.8068 & 1.7983 & 1.7952 & 1.29 & 1.7928 \\ 
& $\omega_{h3}$ & 1.8776 & 1.8534 & 1.8469 & 1.8453 & 1.90 & 1.8446 \\
& $\omega_{h4}$ & 3.0401 & 2.9796 & 2.9593 & 2.9530 & 1.60 & 2.9497 \\ \hline \hline
\multirow{4}{*}{$64$} & $\omega_{h1}$ & 0.7138 & 0.7019 & 0.6962 & 0.6939 & 1.12 & 0.6916 \\
& $\omega_{h2}$ & 1.8300 & 1.8083 & 1.7988 & 1.7954 & 1.27 & 1.7926 \\
& $\omega_{h3}$ & 1.8806 & 1.8543 & 1.8471 & 1.8453 & 1.89 & 1.8446 \\ 
& $\omega_{h4}$ & 3.0493 & 2.9828 & 2.9604 & 2.9534 & 1.59 & 2.9496 \\ \hline
\end{tabular}
\end{center}
\caption{Lowest four approximated frequencies and convergence orders for  $\mathcal{T}_{h}^{1}$ and $\nu =$ 0.45.}
\end{table}

\subsection{Orders of convergence}

Now we are interested in the computation of convergence orders for the eigenvalues of problem
\eqref{eq:mixed_boundary_system}.
For the computation of the spectrum we consider  \eqref{classicspec} as stabilization term,  which we have scaled with the parameter  $\alpha := \text{tr}(a_{h}(\cdot,\cdot))/2$. The meshes for this test are the following:
\begin{itemize}
\item $\mathcal{T}_{h}^{1}$: Deformed triangles with middle points,
\item $\mathcal{T}_{h}^{2}$: Deformed squares.
\end{itemize}


For this test in particular we consider the physical parameters of steal: Young modulus $\Lambda =$ 1.44 $\times 10^{11}$ Pa  and density $\varrho =$ 7.7 $\times 10^{3}$ $kg/m^{3}$. Also, as Poisson ratio we consider  $\nu = 0.35$. On the other hand,  to perform the numerical method, we consider as $N$ the number of polygons that yield on the clamped side of the square.

In Table \ref{tabla_order_mixed} we report the computed eigenfrequencies for $\mathcal{T}_{h}^{1}$, $\mathcal{T}_{h}^{1}$, and different refinement parameter, together with the corresponding extrapolated frequencies and the extrapolated values obtained in  \cite{MR4050542} for a standard VEM.

\begin{table}[H]
\begin{center}
\begin{tabular}{|c|c|c|c|c|c|c|c|c|} \hline

& Mesh & $\mathrm{N} = 16$ & $\mathrm{N} = 32$ & $\mathrm{N} = 64$ & $\mathrm{N} = 128$ & Order & Ext. & \cite{MR4050542} \\ \hline
\hline
$\omega_{h1}$ & \multirow{6}{*}{$\mathcal{T}_{h}^{1}$} & 2957.193 & 2949.107 & 2946.023 & 2944.964 & 1.42 & 2944.259 & 2944.387 \\
$\omega_{h2}$ & & 7363.191 & 7354.174 & 7350.750 & 7349.555 & 1.42 & 7348.775 & 7348.674 \\
$\omega_{h3}$ & & 7902.414 & 7885.866 & 7881.655 & 7880.587 & 1.98 & 7880.231 & 7879.746 \\
$\omega_{h4}$ & & 12805.665 & 12761.802 & 12750.971 & 12748.230 & 2.01 & 12747.348 & 12746.013 \\
$\omega_{h5}$ & & 13119.363 & 13071.579 & 13057.764 & 13053.773 & 1.79 & 13052.114 & 13051.220 \\
$\omega_{h6}$ & & 14948.578 & 1.4905.390 & 14894.439 & 14891.575 & 1.97 & 14890.626 & 14889.584 \\ \hline
$\omega_{h1}$ & \multirow{6}{*}{$\mathcal{T}_{h}^{2}$} & 2987.630 & 2960.174 & 2949.954 & 2946.460 & 1.45 & 2944.256 & 2944.387 \\
$\omega_{h2}$ & & 7394.495 & 7366.040 & 7355.148 & 7351.258 & 1.41 & 7348.770 & 7348.674 \\
$\omega_{h3}$ & & 7971.756 & 7904.912 & 7886.531 & 7881.835 & 1.88 & 7879.896 & 7879.746 \\
$\omega_{h4}$ & & 13041.333 & 12823.521 & 12766.678 & 12752.180 & 1.94 & 12746.809 & 12746.013 \\
$\omega_{h5}$ & & 13256.318 & 13115.099 & 13071.314 & 13058.141 & 1.70 & 13052.125 & 13051.220 \\
$\omega_{h6}$ & & 15185.872 & 14968.304 & 14910.842 & 14895.837 & 1.92 & 14890.267 & 14889.584 \\ \hline
\end{tabular}
\end{center}
\caption{\label{tabla_order_mixed} Lowest six  approximated frequencies and convergence orders for the elasticity spectral problem with the parameters of steal.}
\end{table}

Finally, in Figures \ref{fig:meshesxxxx} and \ref{fig:meshesxxxxy} we present plots that represent some of the eigenfunctions for the elasticity eigenproblem with mixed boundary conditions.

\begin{figure}[H]
	\begin{center}
			\centering\includegraphics[height=6.8cm, width=7.4cm]{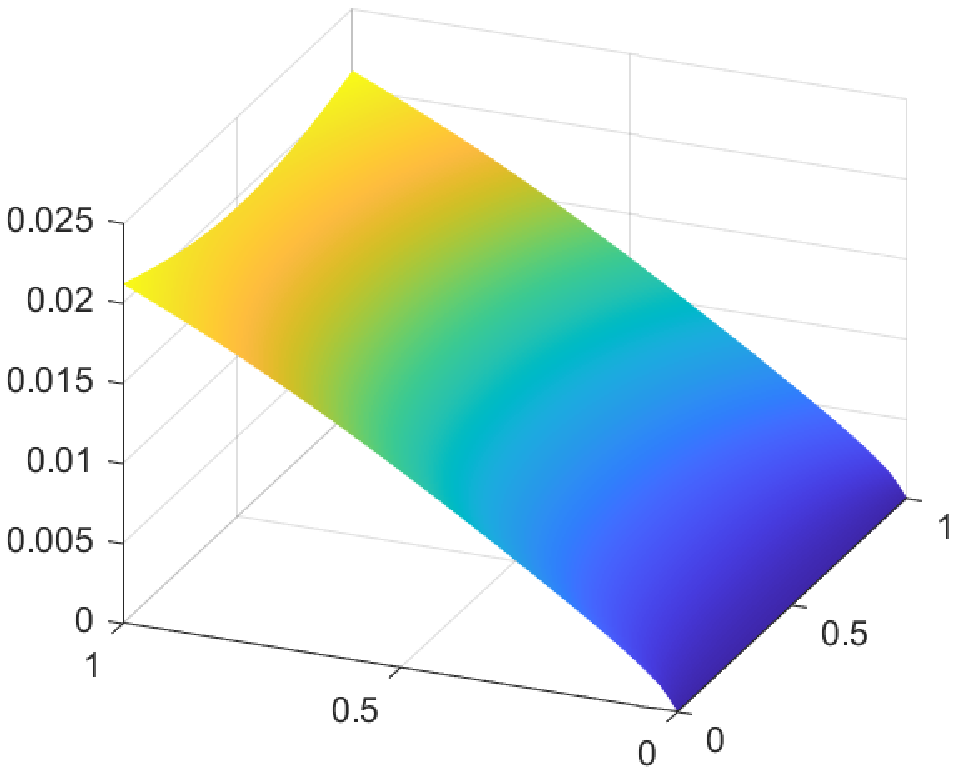}
			\centering\includegraphics[height=6.8cm, width=7.4cm]{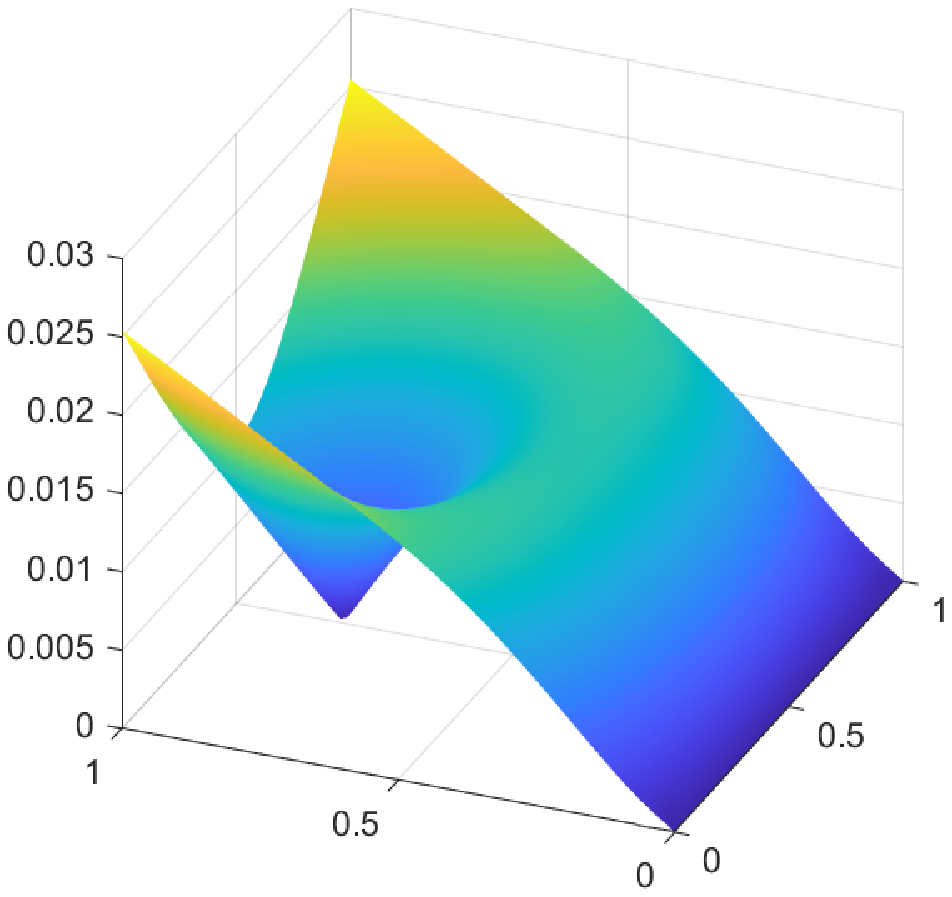}
					\caption{Plots of some approximated eigenfunctions. Left: $\textbf{w}_{h1}$; right: $\textbf{w}_{h3}$.}
		\label{fig:meshesxxxx}
	\end{center}
\end{figure}
\begin{figure}[H]
\begin{center}
			\centering\includegraphics[height=6.8cm, width=7.4cm]{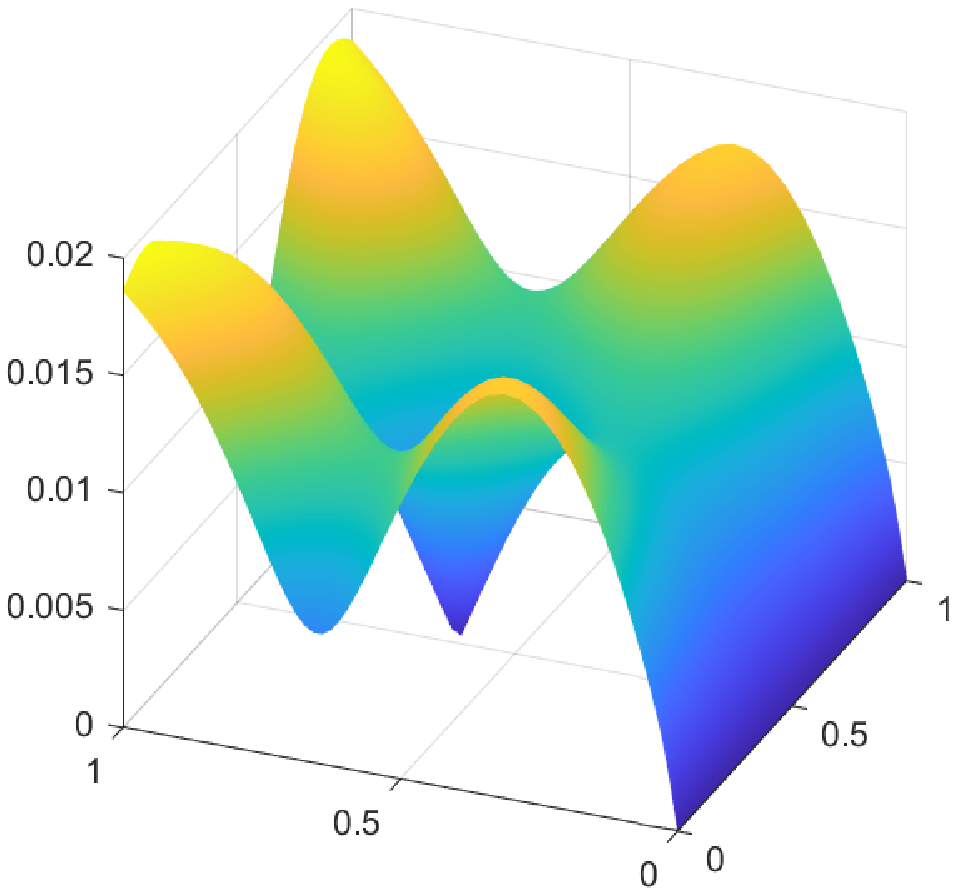}
			\centering\includegraphics[height=6.8cm, width=7.4cm]{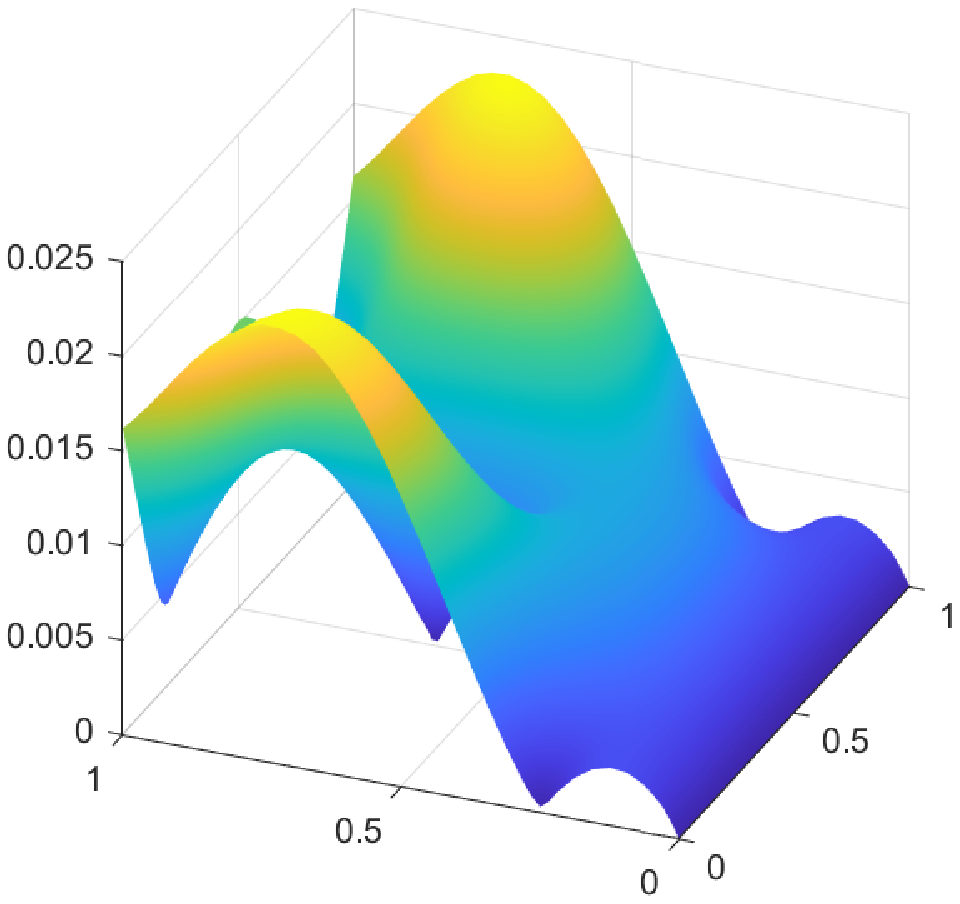}
		\caption{Plots of some approximated eigenfunctions. Left: $\textbf{w}_{h5}$; right: $\textbf{w}_{h6}$.}
		\label{fig:meshesxxxxy}
	\end{center}
\end{figure}

\bibliographystyle{siam}
\footnotesize
\bibliography{AmLeRi_eigen}

\end{document}